\documentclass[reqno]{amsart}
\usepackage{epsfig,latexsym,amsfonts,amssymb,amsmath,amscd}
\usepackage{amsthm}
\usepackage[mathscr]{eucal}
\usepackage{mathrsfs}
\usepackage{mathbbol}
\usepackage{oldgerm,units}
\usepackage{ifthen}

\usepackage{epsfig,latexsym,amsfonts,amssymb,amsmath,amscd,graphics,epic}
\usepackage{amsfonts,amssymb,amsmath,amscd,amsthm}
\usepackage[mathscr]{eucal}
\usepackage{mathrsfs}
\usepackage{oldgerm,units}
\usepackage{wrapfig,epsfig}

\usepackage{ifthen}
\usepackage{mathbbol}

\usepackage{amsthm}
\usepackage[mathscr]{eucal}
\usepackage{mathrsfs}
\usepackage{mathbbol}
\usepackage{oldgerm,units}
\usepackage{wrapfig}

\newtheorem{theorem}{Theorem}[section]

\newtheorem{definition}[theorem]{Definition}

\newtheorem{example}[theorem]{Example}
\newtheorem{remark}[theorem]{Remark}

\newtheorem{note}[theorem]{Note}

\newcommand{\Ref}[1]{(\ref{#1})}

\newcommand{\Real}{\mathbb R}

\newcommand{\Net}{\mathbb N}

\newcommand{\Trop}{\mathbb T}

\newcommand{\one}{\mathbb{1}}
\newcommand{\zero}{\mathbb{0}}

\newcommand{\trop}[1]{\mathcal{#1}}

\newcommand{\tB}{\trop{B}}

\newcommand{\tG}{\trop{G}}
\newcommand{\tH}{\trop{H}}

\newcommand{\tT}{\trop{T}}

\newcommand{\To}{\longrightarrow }

\newcommand{\Hom}{Hom}

\newcommand{\eps}{\varepsilon}
\newcommand{\ep}{\epsilon}
\newcommand{\al}{\alpha}
\newcommand{\bt}{\beta}
\newcommand{\gm}{\gamma}

    \ifx\proof\undefined
    \newenvironment{proof}{
    \smallskip
    \noindent\emph{Proof.}}{\hfill\(\Box\)
    \bigskip
    } \fi

\newcommand{\ifdef}[3]{\ifthenelse{\equal{#1}{true}}{#2}{#3}}

\def\gperp{ {\perp \joinrel  \joinrel \joinrel  \perp }}

\def\({\left(}
\def\){\right)}
\def\atangible{almost tangible}
\def\um{I}

\def\bR{\overline R}

\def\nb{\nabla}

\def\Bnb{{\Bnu{\nabla}}}

\def\otB{\overline{\tB}}

\def\oA{A_{\tB}}
\def\trn{{\operatorname{t}}}
\def\tilL{\widetilde{L}}
\def\tidep{tropically independent}

\def\crit{C}

\def\gen{S}

\def\tdep{tropically dependent}
\def\tdepc{tropical dependence}

\newcommand{\Inu}[1]{\widehat{#1}}
\newcommand{\Det}[1]{ \left|{#1}\right|}
\def\Rd{R^\dagger}
\def\nb{\nabla}

\def\Bnb{{\Bnu{\nabla}}}

\def\hnu{\hat \nu}

\newcommand{\Bnu}[1]{\overline{#1}}

\def\Gker{\operatorname{g-ker}}

\def\semiring{semiring}
\def\semirings{semirings}

\def\semiring0{semiring$^\dagger$}
\def\semirings0{semirings$^\dagger$}
\def\domain0{domain$^\dagger$}
\def\domains0{domains$^\dagger$}
\def\semifield0{semifield$^\dagger$}
\def\semifields0{semifields$^\dagger$}

\def\nucong{\cong_\nu}

\def\nug{>_\nu}
\def\nul{<_\nu}

\def\nuge{\ge_\nu}

\newcommand{\etype}[1]{\renewcommand{\labelenumi}{(#1{enumi})}}
\def\eroman{\etype{\roman}}
\def\hnu{\hat{\nu}}
\def\pipeGS{{\underset{\operatorname{\, gs }}{\mid}}}

\def\lmod{\mathrel   \pipeGS \joinrel\joinrel \joinrel =}
\def\lmodh{\lmod}

\def\lmodg{\lmod}

\def\lmodh{\lmod}

\def\nlmodh{\mathrel  \pipeGS \joinrel \! \joinrel \neq}

\def\ggsim{\, \, \curlyvee \,}
\def\gsim{ {\underset{{gd}}{\ggsim }} }
\def\hsim{ {\underset{{gd}}{\ggsim}}}

\def\tilG{\widetilde{G}}

\def\ghost{\operatorname{ghost}}

\def\nb{\nabla}
\def\ealph{\etype{\alph}}

\def\VHmu{(V,\tHz)}

\def\RGznu{(R,\tGz,\nu)}

\def\RGnu{(\Rd,\tG,\nu)}

\def\base{\tB}
\def\hnu{\hat{\nu}}
\def\nb{\nabla}

\def\Skip{\vskip 1.5mm}
\def\pSkip{\vskip 1.5mm \noindent}
\def\tmap{\varphi}
\def\invr{{\operatorname{-1}}}
\def\Ann{{\operatorname{Ann}}}
\def\rad{\operatorname{rad}}

\def\rank{\operatorname{rk}}

\def\diag{\operatorname{diag}}
\def\a{\alpha}
\newtheorem{thm}[theorem]{Theorem}
\newtheorem*{thm*}{Theorem}

\newtheorem{cor}[theorem]{Corollary}

\def\dper{generalized permutation}
\def\Hom{\operatorname{Hom}}
\newtheorem{lem}[theorem]{Lemma}
\newtheorem{rem}[theorem]{Remark}
\newtheorem{prop*}{Proposition}

\newtheorem{prop}[theorem]{Proposition}
\newtheorem{defn}[theorem]{Definition}
\newtheorem*{examp*}{Example}
\newtheorem*{examples*}{Examples}
\newtheorem*{remark*}{Remark}
\newtheorem*{defn*}{Definition}

\def\lfun{\ell}
\def\R{\Real}

\def\tT{\mathcal T}

\def\tTz{\tT_\zero}
\def\Rz{R}
\def\Fz{F}

\numberwithin{equation}{section}

\def\M0{M_{\zero}}

\def\perm{\widetilde{ P}}

\def\SR{R}

\def\tGz{\mathcal G_\zero}
\def\tHz{\mathcal H_\zero}

\def\PS{P}

\def\Sp{S^\gperp}

\def\rzero{\zero_\SR}

\def\rone{{\one_\SR}}

\def\rzero{\zero_\SR}

\def\vzero{\zero_V}

\def\fone{\one_F}
\def\fzero{\zero_F}

\newcommand{\nPS}[1]{\PS_{(!#1)}}
\newcommand{\nPSo}[1]{\nPS{\one}}

\newcommand{\bil}[2]{\langle{#1},{#2}\rangle}

\newcommand{\adj}[1]{\operatorname{adj}({#1})}

\begin{document}


\title[Supertropical linear algebra]
{Supertropical linear algebra}

\author[Z. Izhakian]{Zur Izhakian}
\address{Department of Mathematics, Bar-Ilan University, Ramat-Gan 52900,
Israel}\email{zzur@math.biu.ac.il}

\author[M. Knebusch]{Manfred Knebusch}
\address{Department of Mathematics, University of Regensburg, Regensburg,
Germany} \email{manfred.knebusch@mathematik.uni-regensburg.de}

\author[L. Rowen]{Louis Rowen}
\address{Department of Mathematics, Bar-Ilan University, Ramat-Gan 52900,
Israel} \email{rowen@macs.biu.ac.il}

\thanks{This research of the first and third authors is supported  by the
Israel Science Foundation (grant No.  448/09).}

\thanks{The second author was supported in part by the Gelbart Institute at
Bar-Ilan University, the Minerva Foundation at Tel-Aviv
University, the Mathematics Dept. of Bar-Ilan University, and the
Emmy Noether Institute.}

\thanks{Research on this paper was carried out by the three authors in the Resarch
in Pairs program of the MFO in Oberwohlfach, July 2009.}

\thanks{This research of the first author has been supported  by the
Oberwolfach Leibniz Fellows Programme (OWLF), Mathematisches
Forschungsinstitut Oberwolfach, Germany.}



\date{\today}




\begin{abstract}

 The objective of this paper is to lay out the algebraic
 theory of supertropical vector spaces and linear algebra,
 utilizing the key antisymmetric relation of ``ghost surpasses.''
 Special attention is paid to the various notions of ``base,'' which include
 d-base and s-base, and these are compared to other treatments in the tropical theory. Whereas
  the number of elements in a d-base
may vary according to the d-base, it
 is shown that when
an s-base exists, it is unique up to permutation and
multiplication by scalars, and can be identified with a set of
``critical'' elements.  Linear functionals and the dual space are
also studied, leading to supertropical bilinear forms and a
supertropical version of the Gram matrix, including its connection
to linear dependence, as well as a supertropical version of a
theorem of Artin.
\end{abstract}

\maketitle




\section{Introduction}
\numberwithin{equation}{section}

The objective of this paper is to lay out an algebraic theory for
linear algebra in tropical mathematics. Extending  the max-plus
algebra to the supertropical algebra of
\cite{IzhakianRowen2007SuperTropical} (which was designed as an
algebraic foundation for tropical geometry),  we obtain a theory
paralleling the classical structure theory of commutative
algebras.

Although there already is an extensive literature on tropical
linear algebra over the max-plus algebra, including  linear
dependence~\cite{AGG} and matrix rank~\cite{ABG}, the emphasis
often is  combinatoric or geometric. The traditional approach in
semiring theory is to divide the determinant into a positive and
negative part (since $-1$ need not exist in the semiring),
cf.~\cite{Plus}. Whereas this approach provides many basic
important properties of matrices, such as a general method given
in \cite{AGG} to transfer identities from ring theory to semiring
theory, the reliance on combinatorics also leads to competing (and
different) definitions. For example, in \cite{ABG}, five different
definitions of matrix rank are given: The row rank, the Barvinok
(Shein) rank, the strong rank, the Gondran-Minoux rank, the
symmetrized rank, and the Kapranov rank.

The structure theory of supertropical semirings tends to unify
these notions, giving a single formula for the determinant, from
which we can define a nonsingular matrix; in this approach, the
row rank, column rank, and strong rank all coincide. This makes it
easier to proceed with a traditional algebraic development.
Explicitly, properties of matrices were studied in
\cite{IzhakianRowen2008Matrices},
\cite{IzhakianRowen2009Equations},
\cite{IzhakianRowen2010MatrixIII}, and
\cite{IzhakianRowen2009TropicalRank}, where the main
 theme is to replace the max-plus algebra by a cover, called the
 \textbf{supertropical semiring}, which permits one to formulate stronger results
 which are amenable to proofs more in line with classical matrix theory.  Recall that the underlying
supertropical structure is a semiring (without zero), $R$,~with a
designated semiring ideal $\tG \supseteq mR$ for all $m$, where
$mR$ denotes $a+ \cdots + a$ repeated $m$ times; the algebraic
significance is obtained by interpreting $\tG$ as ``ghost
elements'',  elements which collectively are treated analogously
to a zero element. When convenient, one assumes that $R$ contains
a zero element $\rzero,$ which can be formally adjoined. Thus, we
introduce the fundamental relation $a \lmodg b$ when $a$ equals
$b$ plus a ghost element (which could be $\rzero$).

 We recall that the
\textbf{tropical determinant} of an $n \times n$  matrix $A =
(a_{i,j})$ in $M_n(\Rz )$ is really the permanent, which we denote
as
\begin{equation*}\label{det2}
|A | = \sum _{\pi \in S_n}a_{\pi (1),1} \cdots a_{\pi
(n),n}.\end{equation*}

Although the equation $|AB| = |A||B|$ fails over the max-plus
algebra, the relation $|AB| \lmodg |A||B|$ holds over a
supertropical semiring,
\cite[Theorem~3.5]{IzhakianRowen2008Matrices}, and any matrix
satisfies its characteristic polynomial in the sense of
\cite[Theorem~5.2]{IzhakianRowen2008Matrices}. The roots of this
polynomial are precisely the supertropical eigenvalues.

 Our main objective here is to initiate a formal
theory of supertropical vector spaces and their bases, over
semirings with ghosts, and in particular over supertropical
semifields.

Our method is to rely as far as possible on the structure theory.
While this theory parallels the classical theory of linear
algebra, several key differences do emerge. At the outset, one
major difference is that there are two different kinds of bases.
First, one can take a maximal (tropically) independent set, which
we call a \textbf{d-base}, called a ``basis'' in
\cite[Definition~5.2.4]{MS09}. This has considerable geometric
significance, intuitively providing a notion of rank (although, by
an example in \cite{MS09}, the rank might vary according to the
choice of d-base). As one might expect from
\cite{IzhakianRowen2009Equations}, any dependence among vectors
can be enlarged to an (often unique) \textbf{saturated
dependence}, which is maximal in a certain sense;
cf.~Theorem~\ref{sat1}. This leads to a delicate analysis of rank
of a subspace, especially since it turns out that the number of
elements in different d-bases may differ.

Alternatively, one can consider sets that (tropically) span the
subspace;    an \textbf{s-base} is a minimal such set when it
exists. Such sets are used in generating convex spaces, as studied
in \cite{CG}. Not every d-base is an s-base. In fact, the number
of elements of an s-base might necessarily be larger than the
number of elements of a d-base. Surprisingly, an s-base is unique
up to scalar multiples , and can be characterized in terms of
\textbf{critical} elements, which intuitively are elements that
cannot be decomposed into sums of other elements. On the other
hand, the d-bases can be quite varied, and lead us to interesting
subspaces that they span, which we call \textbf{thick}.

 We also consider linear
transformations in this context, in which the  equality $\tmap
(v+w) = \tmap (v) + \tmap(w)$ is replaced by the ghost surpassing
relation $\tmap (v+w) \lmodh \tmap (v) + \tmap(w)$. Linear
transformations lead us to the notion of the \textbf{dual space}.
The dual space depends on the choice $\tB$ of d-base, but there is
a natural ``dual s-base'' of the dual space of  $\tB$, of the same
rank (Theorem~\ref{dualbasethm}).

In the last section we introduce supertropical bilinear forms, in
order to study ``ghost orthogonality'' between vectors. One
calls two vectors $v$ and $w$
 \textbf{g-orthogonal} with respect to a supertropical bilinear form
$\bil{\phantom{w}}{\phantom{w}}$ when $\bil vw$ is a ghost.
 We construct the \textbf{Gram matrix} and prove the
connection between tropical dependence of vectors in a
nondegenerate space and the singularity of this matrix (Theorem~
\ref{thm:GramMat}). Finally, we prove  (Theorem~\ref{orthogsym}) a
variant of Artin's Theorem: When the g-orthogonality relation is
symmetric, the supertropical bilinear form is ``supertropically
symmetric.''

Since the exposition \cite{MS09} is an excellent source of
fundamental results and examples, we use it as a general reference
for the ``standard'' tropical theory   and  compare several of our
definitions with the definitions given there.

\section{Supertropical structures}

\subsection{Semirings without zero}

 A \textbf{semiring without zero}, which we notate as \semiring0,
is
 a structure $(\R ,+,\cdot, \rone)$ such that $(\R ,\cdot \,
,\rone)$ is a monoid and $(\R ,+)$ is a commutative semigroup,
with distributivity of multiplication over addition on both sides.
(In other words, a \semiring0 does not necessarily have the zero
element $\zero$, but any semiring can also be considered as a
\semiring0.)

The reason one does not initially require a zero element is
twofold: On the one hand, in contrast to ring theory, the zero
element plays at best a marginal role in semirings, because of the
lack of additive inverses, and often gets in the way, requiring
special treatment in definitions and propositions; on the other
hand, in our main example, the max-plus algebra of $\Real$, the
zero element does not exist in $\Real$ but is adjoined formally
(as $-\infty$), and often gets in the way (for example, when one
wants to evaluate Laurent polynomials.) At any rate, given a
\semiring0 $\Rd$, we can formally adjoin the element $\zero $ to
obtain the semiring $\Rz : = \Rd \cup \{ \zero  \},$ where we
stipulate for all $a \in \Rz$:
$$\zero +a = a+ \zero  = a; \qquad \zero a = a\zero  = \zero .$$

A
 \textbf{\semiring0 with ghosts} is a
triple $\RGnu,$ where $\Rd$ is a \semiring0    and $ \tG   $ is a
\semiring0 ideal, called the \textbf{ghost ideal}, together with
an idempotent map
$$\nu : \Rd \To \tG$$  called the \textbf{ghost map on} $\Rd$,
given by $$\nu (a) = a+a.$$ (We require that $\nu$ preserves
multiplication as well as addition.) We write $a^\nu$ for
$\nu(a).$ Thus, $$e:= \rone^\nu$$ is both a multiplicative and
additive idempotent of $\Rd$, which plays a key role since
$\nu(\Rd) = e\Rd.$

 A \textbf{supertropical \semiring0}  has the extra properties:
\begin{enumerate} \ealph
 \item  $a+b   =  a^{\nu } \quad \text{if}\quad a^{\nu } =
 b^{\nu}$; \pSkip
 \item $a+b  \in \{a,b\},\ \forall a,b \in \Rd \ s.t. \  a^{\nu }
\ne b^{\nu }.$

 (Equivalently, $\tG$ is ordered, via $a^\nu \le
b^\nu$ iff $a^\nu + b^\nu = b^\nu$.) \pSkip
\end{enumerate}

We write $a \nug b$ if $a^{\nu } >
 b^{\nu}$; we stipulate that $a$ and
$b$ are $\nu$-{\bf matched}, written $a \nucong b$,  if $a^{\nu }
=
 b^{\nu}$. We say that  $a$  {\bf dominates} $b$ if  $a \nug  b$.

Recall that any commutative supertropical semiring satisfies the
\textbf{Frobenius formula} from
\cite[Remark~1.1]{IzhakianRowen2007SuperTropical}:
\begin{equation}\label{eq:Frobenius} (a+b)^m = a^m + b^m
\end{equation}
for any
    $m \in \Net^+$.

 A \textbf{supertropical \domain0}  \cite{IzhakianRowen2007SuperTropical} is
 a supertropical \semiring0 $\Rd$ for which  $$\tT : = \Rd \setminus \tG$$  is a
 multiplicative monoid, such that the map $\nu
|_\tT : \tT \to \tG$ (defined as the restriction from $\nu$ to
$\tT$) is onto. $\tT$ is  called the set of \textbf{tangible
elements} of $\Rd$.
 A \textbf{supertropical \semifield0} is a
supertropical \domain0 $\RGnu$ in which $\tT$ is a group. Thus,
$\tG$ is also a group. \Skip

We have the analogous definitions when we adjoin the element
$\rzero$ to the \semiring0 $\Rd$ to obtain the \textbf{semiring
with ghosts} $\Rz$. Thus, we write $$\Rz := \Rd\cup \{ \rzero \} =
\RGznu,$$ where  $ \tGz := \tG \cup \{ \rzero \} $ is a semiring
ideal, called the \textbf{ghost ideal}, and the ghost map  $\nu :
\Rz \to \tGz$ satisfies $\nu (\rzero) = \rzero.$ Conversely, given
a semiring with ghosts $ \RGznu,$ we can take $\Rd = \Rz \setminus
\{ \rzero \}$ and $\tG = \tGz \setminus \{ \rzero \}$ and define
the \semiring0 with ghosts $\RGnu.$ Thus, the theories with or
without $\rzero$ are basically the same.

 In this spirit,
we say that $\Rz$ is a
  \textbf{supertropical semiring} when $\Rd$ is a supertropical
  \semiring0, and say that $\Rz$ is a
  \textbf{supertropical domain} when $\Rd$ is a supertropical \domain0; i.e.,  $\tT   = R \setminus \tGz$  is
  the
   monoid  of
 tangible elements. (We
  write $\tTz$ for $\tT \cup
\{ \rzero \} $ in the supertropical domain $\Rz$.) Likewise, a
  \textbf{supertropical semifield}  is a
supertropical domain $\RGznu$ in which $\tT$ is a group.

Intuitively, the tangible elements correspond in some sense to the
original max-plus algebra, although here $a+a = a^\nu$ instead of
$a+a = a.$ Our motivating example of supertropical semifield, used
as the primary example throughout
\cite{IzhakianRowen2007SuperTropical} as well as in this paper, is
the extended tropical semiring \cite{zur05TropicalAlgebra}
$$\Trop := D(\Real) := (\Real
\cup \Real^\nu \cup \{-\infty \} , \Real^{\nu} \cup \{-\infty \},
1_\Real),$$  the most familiar example of a supertropical
semifield whose operations are induced by the standard operations
$\max$ and $+$ over the real numbers; we call this
\textbf{logarithmic notation}, since the zero element
$\zero_\Trop$ is~$-\infty$ and the unit element $\one_\Trop$ is
$0$.

The supertropical domain, and in particular the supertropical
semifield, seem to play a basic role in supertropical algebra
parallel to the role of the field in classical algebra. In
\cite{IzhakianRowen2007SuperTropical} a reduction is given from
supertropical domains  to supertropical semifields. Accordingly,
one is led to study linear algebra over supertropical semifields.

  Occasionally, we also want to pass back from $\tG$ to $\tT.$
Abusing notation slightly, we pick a representative in $\tT$ for
each class  in the image of $\hnu$, thereby getting a function
 $$\hnu: \Rd \to \tT $$ by putting $\hnu |_{\tT} = 1_{\tT};$
 also, by definition,   $\nu \circ ( \hnu |_{\tG} )= 1_{\tG}.$  In this case, we also write $\hat a$ for
$\hnu(a),$ but when the notation becomes cumbersome, we still use
the $\hnu$ notation.

  Here are two reductions to the case  that $\nu_\tT$ is 1:1.

\begin{rem}\label{oneone} We define an equivalence on $R$ via $a\equiv b$ when either $a=b$ or $a,b \in  \tT$ with  $a
\nucong b.$ In other words, two tangible elements are equivalent
iff they are $\nu$-matched. Then we could define the supertropical
\domain0 $$\bR := \Rd /_\equiv $$ to be $ \ (\tT/_\equiv)\cup
\tG.$ The ghost map $\nu$ defines a 1:1 function from the
equivalence classes of $\tT$ to $\tG$. \end{rem}

\begin{rem}\label{onetwo} In \cite[Proposition~1.6]{IzhakianRowen2009Equations}
we see that $\hat\nu$ can be chosen to be multiplicative on $\tG$.
When $\tG$ is a multiplicative group,  define
$$\widetilde T := \hat \nu(\tG) = \{ a \in \tT: \widehat{a^\nu }=
a\},$$  and
$$R' := \widetilde T \cup \tG,$$
and let $\nu'$ be the restriction of $\nu$ to $R'$. Then $(R',
\tG, \nu')$ is a supertropical \domain0, whose tangible elements
are $\widetilde T $, and $\nu'|_{\widetilde T}: \widetilde T \to
\tG$ is 1:1.
\end{rem}

\begin{rem}\label{oneonetwo} When $\nu_\tT$ is not 1:1, it is
convenient to define $$\tT_e := \{ a \in \tT: a \nucong \rone
\}.$$ Note that $\tT_e$ is a submonoid of $\tT,$ and in fact
$\tT_e \cup \{ e\}  $ is a supertropical \domain0\ contained
in~$R$.
\end{rem}


 To clarify our exposition,  most of the examples in this
paper are presented for the extended tropical semiring $D(\Real)$.

\subsection{The  ``ghost
surpass'' and ``ghost dependence'' relations} 

We consider the semiring with ghosts $\RGznu.$

\begin{definition} We say $b$ is \textbf{ghost dependent} on $a$, written $b \gsim a$,  if
$a+b \in \tGz. $ \end{definition}
 \noindent In particular, $a \nucong b$  implies that $a \gsim
b$.

Note that the ghost dependence relation is symmetric, but not
transitive, since $1 \gsim 3^\nu$ and $3^\nu \gsim 2,$ although
$1$ and~$2$ are not ghost dependent. The following antisymmetric
and transitive relation is a key to much of the theory.

\begin{definition} We define the \textbf{ghost
surpasses} relation $\lmodg$ on   $\Rz $, by $$a \lmodg b \qquad
\text{ iff } \qquad a = b + c \quad \text{for some  } \quad c\in
\tGz.$$
\end{definition}
\noindent  In this notation, by writing $a \lmodg \rzero$ we mean
$a \in \tGz$.
This restricts to the \textbf{ghost surpasses} relation on
$\Rd$, by $$a \lmodg b \qquad \text{ iff } \qquad a = b \qquad
\text{or} \qquad a = b + c \quad \text{for some  } \quad c\in
\tG.$$

\begin{rem}\label{rem:ghostSrp} The following are equivalent:

\begin{enumerate}
\item
$a \gsim \rzero$; \pSkip  \item $a \in \tGz$; \pSkip \item  $a
\lmodg \rzero$.
\end{enumerate}
\end{rem}

We quote some easy properties of $\lmodg$ from
\cite{IzhakianRowen2009Equations}:
\begin{rem}  \label{surpass1} $ $
\begin{enumerate}\eroman
  \item (\cite[Remark 1.2]{IzhakianRowen2009Equations}) When    $a$ is tangible, $a \lmodg b$ implies that $a=b.$
In particular, tangible elements are comparable under $\lmodg$ iff
they are equal. In this way, the relation $\lmodg$ generalizes
equality. \pSkip

\item $a \lmodg b$ iff $a=b$ or $a$ is a ghost $\nuge b.$ In
particular, if $a \lmodg b$ then $a \nuge b;$ if $a \lmodg b$ for
$b \in \tGz$, then $a \in \tGz$. \pSkip

\item (\cite[Lemma 1.5]{IzhakianRowen2009Equations}) $\lmodg$ is
an antisymmetric partial order on $R$. \pSkip

 \item If $a \lmodg b$, then $a \gsim b.$

 \end{enumerate}

 \end{rem}

  \begin{lem} \label{surplem} Generalizing Remark~\ref{surpass1}(i),
   for $R$ a supertropical \domain0,
   an element $a \in R$ is tangible iff  the following condition holds:

   $a \lmodg
  b$ implies $b = \a a$ for some $\a \in \tT_e.$
 \end{lem}
 \begin{proof} $(\Rightarrow)$ is by Remark~\ref{surpass1}(i).
 Conversely, suppose  $a$ is not tangible; i.e. $a  \in \tG$, so $a=
 a^\nu.$
  Then $a \lmodg \hat a,$ where $\hat a\in \tT$ and  $(\hat a)^\nu = a.$
The condition implies $\hat a = \a a$ for some  $\a \in \tT_e,$
which is impossible since $\a a \in \tG.$
 \end{proof}

This leads us later to a good abstract criterion for tangibility.
Also, conversely to Remark~\ref{surpass1}(iv), we have

 \begin{lem}\label{surpass2} $ $ \begin{enumerate}\eroman
  \item  If $a \gsim b$ with $b\in \tT$,
 then either $a \equiv b$ or  $a \lmodg b$. \pSkip

\item If $a \gsim b$ with $a \nuge b,$
 then either $a \equiv b$ or $a \lmodg b$.  \end{enumerate} \end{lem}
 \begin{proof} (For both parts) If $a \in \tG$ with $a \nuge
 b,$ then $a = b+a.$ So we are done unless $a \in \tT$, which implies $a \nucong b$, and thus
 $a = b$, since $\nu |_\tT$ is assumed to be 1:1.
 \end{proof}

\subsection{Vector spaces with ghosts}
Modules over semirings (often called ``semimodules'' in the
literature \cite{qtheory}, or sometimes ``cones'') are defined
just as modules over rings, except that now the additive structure
is that of a semigroup instead of a group. (Note that subtraction
does not enter into the other axioms of a module over a ring.)

\begin{defn}\label{def:module0} Suppose $R$ is a semiring.
An \textbf{$R$-module} $V$ is a semigroup $(V,+,\vzero)$ together
with scalar multiplication $R\times V \to V$ satisfying the
following properties for all $r_i \in R$ and $v,w \in V$:
\begin{enumerate}
    \item $r(v + w) = rv + rw;$ \pSkip
    \item $(r_1+r_2)v = r_1v + r_2 v;$ \pSkip
    \item $(r_1r_2)v = r_1 (r_2 v);$ \pSkip
    \item $\one_R v =v;$ \pSkip
    \item $r \zero_V = \zero_V$; \pSkip
  \item $  \rzero v = \zero_V.$
   \end{enumerate}
 \end{defn}

\begin{note} One could also define module over a \semiring0, by deleting Axiom (6).
In the other direction, any module $V$ over a \semiring0 $\Rd$
becomes an $\Rz$-module when we formally define $  \rzero v =
\zero_V$ for each $v\in V.$
\end{note}

The reason we prefer the terminology ``module'' is that this
definition of module over a semiring~$R$ coincides with the usual
definition of module when $R$ is a ring, since $-v = (-\one_R)v .$
In case the underlying semiring is a semiring with ghosts, $V$ has
the distinguished submodule $e V,$ as well as the ghost map $\nu:
V \to e V,$ given by $$\nu(v) := v+v = (\rone + \rone) v = e v\in
e V.$$

\begin{lem} Any $R$-module   $V$ over a semiring
with ghosts $R$
 satisfies the following properties
 for all $r\in R,$ $v\in V:$
\begin{enumerate}
    \item $(rv)^\nu =r  v^\nu= r^\nu v$; \pSkip
    \item $(v+w)^\nu =v^\nu + w^\nu$. \pSkip
 \end{enumerate}\end{lem}
 \begin{proof} (1) $(rv)^\nu = e( r v) = (e r) v = (r e ) v =  r (e v) = r v^\nu.$

 (2) $(v+w)^\nu = e(v +w)=   e v + e w = v^\nu +
 w^\nu$.\end{proof}

In order to obtain a stronger version of supertropicality we
introduce the following definition:

\begin{defn}\label{def:module} Suppose  $R = \RGznu$ is a semiring  with
ghosts. An \textbf{$R$-module with ghosts} $\VHmu$  is comprised
of an $R$-module $V$ and an R-submodule $ \tHz \supseteq e V $
satisfying the axiom:
\begin{align*}
 & \quad   v^\nu = w^\nu \quad \text{implies} \quad v+w =
v^\nu, \quad \forall v,w \in V .
\end{align*}

We call  $ \tHz$ the  \textbf{ghost submodule} of $V$, and $\nu$
is called the \textbf{ghost map on} $V$.
\end{defn}
We define the map $\nu: V \to \tHz,$ given by $\nu(v) := v+v = e
v$,  and write $v^\nu$ for $\nu(v)$.

The choice of the ghost submodule can be significant. (Note that
$v^\nu$ could differ from $v$ even when $v \in \tHz$.)  The
\textbf{standard ghost submodule} of $V$ is defined as $e V.$  Any
module over a supertropical semiring can be viewed as a module
with ghosts with respect to the standard ghost submodule $e V$; in
this case, we suppress $\tHz$ in the notation.

\begin{defn}\label{tropmod}
An \textbf{$R$-submodule with ghosts} of $(V,\tHz)$ is a submodule
$W$ of $V$, endowed with the ghost submodule $ W\cap \tHz,$ whose
ghost map is the restriction of $\nu$ to $W$.
\end{defn}

  When $R$ is a supertropical semifield, $(V,\tHz)$ is
called a \textbf{(supertropical) vector space} over $R$, or vector
space, for short. We focus on vector spaces in this paper, and
call their elements \textbf{vectors}. A more general investigation
of modules with ghosts is given in
\cite{IzhakianKnebuschRowen2010Modules}.  Our main example of a
vector space in this paper, as well as in
\cite{IzhakianRowen2008Matrices}, is $\Rz ^{(n)}=(\Rz ^{(n)},
\tGz^{(n)})$, whose ghost map acts as $\nu$ on each component. The
zero element $\zero$ of~$\Rz ^{(n)}$ is $(\rzero, \dots,\rzero)$.
 The   tangible vectors   of $\Rz  ^{(n)}$ are     those
$(v_1, \dots, v_n)$ such that each $v_i \in \tTz$.
%
  %




 As with semirings with ghosts,  we define  the
\textbf{ghost surpassing relation $\lmodh$}  for vectors $v,w \in
V$ by:
    $$v \lmodh w \quad \text{  if  } \quad v = w + u \ \text{for some}    \ u  \in
 \tHz.$$

 We say that two vectors $v,w \in V$ are
$\nu$-{\bf matched}, written $v \nucong w$,  if $v^{\nu } =
w^{\nu}$. Likewise, we write $v \nuge w$ if $v^{\nu } = w^{\nu} +
x^\nu$ for some $x^\nu \in \tHz.$

\begin{example}
\begin{equation}\label{value} (v_1, \dots, v_n) \ge _\nu (w_1,
\dots, w_n)\end{equation} in $\Rz ^{(n)}$ iff $v_i \ge _\nu w_i$
for each $1 \le i \le n.$
\end{example}

Also, for elements $v ,w $ in a module with ghosts, we define
 $$v \hsim w \quad \text{ if }  \quad v +w \in \tHz.$$

 \begin{rem}\label{tang2} $ $
 \begin{enumerate}\eroman

 \item If $v\lmodh w$, then $v +w \in \tHz$, i.e., $v \hsim w$.
\item If $v_i  \lmodh w$ for $i = 1,2,$ then $v_1+v_2  \lmodh w$.

\end{enumerate}\end{rem}

\begin{lem}\label{MG1} Any module with ghosts $(V,\tHz)$ satisfies the following property,
for all $v,w\in V,$ $h\in eV:$
  $$ v=w+h \ \Rightarrow \ v+h=v, $$
\end{lem}
\begin{proof} $v =  w+h = w+h+h  =   v+h.$ \end{proof}

\begin{prop}\label{MG10} Any module with ghosts $(V,\tHz)$ satisfies the following property, for all $v,w\in V,$
$h_1,h_2\in \tHz:$
 $$ v+h_1+h_2=v \ \Rightarrow \ v+h_2=v, $$
\end{prop}
\begin{proof} $v =  v+h_1+h_2= (v+h_1+h_2) + (h_1+h_2) = v+eh_1 + eh_2.$ Take $w = v+eh_1$ and $h= eh_2$ in the lemma
to get $ v =  v+eh_2 = ( v+h_1+h_2) + h_2 = v+h_2  .$ \end{proof}

\begin{cor}\label{MG11}
The ghost surpassing relation is always antisymmetric.
\end{cor}


Motivated by Lemma~\ref{surplem}, we have an abstract definition
of tangibility for any vector space over a supertropical semifield
(which is used more generally in
\cite{IzhakianKnebuschRowen2010Modules} for any module over a
supertropical \semiring0):

 \begin{defn} The  \textbf{\atangible\ vectors}  of $V$ are     those
 elements $v\in V$ for which $v \lmodg w$ implies $w \in \tT_e v$,
 $\forall w \in V.$
\end{defn}

\begin{rem}\label{notghost} A nonzero ghost vector $v$ cannot be \atangible, for we  always have
$$v = \bigg(\rone + \frac \rone 2   \bigg)v = v + \frac \rone 2  v \lmodg \frac \rone 2
  v.$$
\end{rem}

\begin{example} Clearly, almost tangible vectors in  $\Rz  ^{(n)}$ are tangible.

On the other hand, in logarithmic notation, taking $R = D(\Real),$
if $V$ is the submodule of $R^{(2)}$ spanned by the the vectors
$v_1 = (1,1^\nu)$ and $v_2 = (0, 1),$ then one sees without
difficulty that $v_1$ is \atangible\ in $V$, although not tangible
in $R^{(2)}$.

In fact, a
 submodule of $\Rz  ^{(n)}$ need not have any tangible vectors
 at all, as exemplified
 by the submodule $R(1, 1^\nu)$ of $R^{(2)}.$
\end{example}

\begin{example}\label{tangvec} For vectors $v = (v_1, \dots, v_n)$ and $w = (w_1,
 \dots, w_n)$ in $\Rz  ^{(n)}$,  $v \lmodh w$ iff
 $v_i \lmodg w_i$ for all $i = 1, \dots, n.$ Thus, checking components, we see that the ghost surpassing relation for vectors
 of $\Rz  ^{(n)}$ is antisymmetric.
\end{example}

Here is another useful property of vectors in $\Rz  ^{(n)}$.

\begin{lem}\label{sumofsat}
If $ v\lmodh \sum _{i=1}^\ell \a_i w_i $ and $v\lmodh \sum
_{i=1}^\ell \a'_i w_i $ in $\Rz  ^{(n)}$, then $v\lmodh  \sum
_{i=1}^\ell \widehat{(\a_i +\a'_i)}  w_i $.
\end{lem}
\begin{proof} Checking components, we may assume that $n=1$. But then the assertion is immediate.
\end{proof}

\section{Background from matrices}\label{ses:matrices}

 Any set $S = \{v_1, \dots,
v_m\}$ of $m$ row vectors in $\Rz  ^{(n)}$ corresponds to an $m
\times n $ matrix $A(S),$ whose $m$~rows are the vectors of $S$.
We call $A(S)$ the \textbf{matrix} of $S$.

We defined $|A|$ in the introduction. We say that the matrix $A$
is \textbf{nonsingular} if $| A |$ is tangible (and thus
quasi-invertible when $R$ is a supertropical semifield
\cite{IzhakianRowen2008Matrices}); otherwise,  $| A | \in \tGz$
(i.e., $|A| \lmodg \fzero$ by Remark \ref{rem:ghostSrp}) and we
say that $A$ is~ \textbf{singular}. In
\cite{IzhakianRowen2008Matrices}, we also defined vectors in $$\Rz
^{(n)}$$ to be \textbf{tropically independent} if no  linear
combination with tangible coefficients is in $\tHz$. By
\cite[Theorem 3.4]{IzhakianRowen2009TropicalRank}, when $R$ is a
supertropical domain,  $A(S)$ has $m$ tropically independent rows
iff $A(S)$ has a nonsingular $m\times m$ submatrix. Thus, it is
natural to try to understand linear algebra in terms of the
supertropical matrix theory of
\cite{IzhakianRowen2008Matrices,IzhakianRowen2009Equations}.

Although it was shown in \cite{IzhakianRowen2008Matrices} that the
product of nonsingular matrices could be singular, we do have the
consolation that the product of nonsingular matrices cannot be
ghost, cf.~Theorem \ref{prdnonsng} below.

 Recall that
a \textbf{quasi-identity} matrix is a nonsingular,
multiplicatively idempotent matrix ghost-surpassing the identity
matrix.  Suppose $A = (a_{i,j}),$ with $\Det{A}$ invertible in
$R$. In \cite[Theorem~2.8]{IzhakianRowen2009Equations}   one
defines the matrix
\begin{equation}\label{eq:Anb}
 A^\nb := \frac \rone {\Det{A}} \adj{A} ,
\end{equation} and obtains the quasi-identity matrices
\begin{equation}\label{eq:IA}
 \um_A = A A^\nb   
; \qquad \um'_A = A^\nb  A. 
\end{equation}

\subsection{Annihilators of matrices}

\begin{defn}
A vector $v\in \Rz  ^{(n)}$ (written as a column)
\textbf{g-annihilates} an $m \times n $ matrix $A$ if $Av \lmodh
\vzero$ in~$\Rz  ^{(n)}$.
Define $$\Ann (A) = \bigg \{ v \in \Rz  ^{(n)}: Av \lmodh \vzero
\bigg \},$$ clearly a submodule of $\Rz ^{(n)}$.\end{defn}
Accordingly, $\tGz^{(n)} \subseteq \Ann (A) $, for any $m \times n
$ matrix $A$.

\begin{rem} $ $
\begin{enumerate} \eroman

    \item
The point of this definition is that the vector
 $v = ( \bt_1, \dots, \bt_m)$ g-annihilates $(A(S))^{\operatorname{t}}$, the transpose of the  matrix of~
  $S = \{w_1, \dots, w_m\},$ iff $\sum _{i=1}^m \bt_i w_i \lmodh \rzero.$
  Thus, tangible g-annihilators correspond to tropical dependence relations.
  \pSkip

\item A (nonzero) tangible vector cannot g-annihilate a
nonsingular matrix,  since the columns are tropically independent.

\end{enumerate}

\end{rem}

We can improve this result, to include vectors that are not
necessarily tangible.

\begin{lem}\label{diag} The diagonal of the product $I_A I_B$ of quasi-identity matrices $I_A,I_B$ cannot all be  ghosts. \end{lem}
\begin{proof} Otherwise, write $I_A = (a_{i,j})$ and $I_B = (b_{i,j})$.
 If the assertion is false, then for each~$i_t$ there is $i_{t+1}$
 such that $a_{i_t,i_{t+1}}b_{i_{t+1},i_t}\ge_\nu \fone.$ Consider
 the digraph $G$ of $I_A I_B$, cf.
 \cite[\S 3.2]{IzhakianRowen2008Matrices}.
By the pigeonhole principle, the path of vertices $i_1, i_2, i_3,
\dots, i_{n+1}$  contains a cycle, say from $i_s$ to $i_s'$. But
the weight of any non-loop cycle in a quasi-identity has
$\nu$-value less than $\rone$.  (Otherwise, multiplying by the
entries $a_{i,i}$ for all vertices $i$ not in the cycle gives an
extra summand $\ge e = \rone^\nu$ for $|I_A|$, contrary to $|I_A|
= \rone.$) Hence
$$\rone \le_\nu \prod_{k=s}^{s'-1} a_{i_k,i_{k+1}} b_{i_{k+1},i_k}
= \prod_{k=s}^{s'-1} a_{i_k,i_{k+1}}\prod_{k=s}^{s'-1}
b_{i_{k+1},i_k} \nul \rone\rone = \rone,$$ a contradiction.
\end{proof}

\begin{thm}\label{prdnonsng} The product of two nonsingular $n \times n$ matrices cannot be in $M_n(\tGz)$. \end{thm}
\begin{proof} If $AB$ is ghost for $A,B$ nonsingular, then in the notation of
\cite[Definition 4.6]{IzhakianRowen2008Matrices},
$$I_{A^\nb}I_B = I'_A I_B = A^\nb A B B^\nb \in M_n(\tGz),$$ contradicting the lemma.
\end{proof}

On the other hand, examples were given in
\cite{IzhakianRowen2008Matrices} in which the product of two
nonsingular $n \times n$ matrices is singular. Here is a related
example using quasi-identities:

\begin{example} The matrices $$A= \left( \begin{matrix} 0 &  0^\nu   \\ - \infty & 0
\end{matrix}\right) , \qquad  B = \left( \begin{matrix} 0 & - \infty  \\ 0^\nu & 0
\end{matrix}\right)$$  over $D(\Real)$ are nonsingular,
but $AB  = \left( \begin{matrix}  0^\nu &  0^\nu   \\  0^\nu & 0
\end{matrix}\right)$ and $BA =  \left( \begin{matrix} 0 &  0^\nu   \\  0^\nu & 0^\nu
\end{matrix}\right)$ are singular.
\end{example}

\section{Tropical dependence}

\emph{ Throughout the remainder of this paper, $F=(F,\tGz, \nu)$
denotes a supertropical semifield.}

 Dependence plays a major role in module theory. For supertropical
 modules with ghosts,
the familiar definition becomes degenerate. The following
modification from \cite{IzhakianRowen2008Matrices},
 in which the role of zero is replaced by the ghost ideal,
 is more suitable for our purposes.

\begin{defn}
Suppose $(V,\tHz)$ is a vector space over ~$F$. A family of
elements $S = \{ w_i: i \in I\}\subset V$ is \textbf{\tdep} if
there exists a nonempty finite subset $I'\subset I$ and a family
$\{\a_i : i \in {I'}\}\subset \tT$, such that
\begin{equation}\label{eq:dependence1}
\sum_{i \in I'} \a_i w_i \in \tHz.\end{equation}
 Any such
relation \eqref{eq:dependence1} is called a \textbf{\tdepc} for
$S$. A subset $S\subset V$ is called \textbf{\tidep} if it is not
tropically dependent.

Given an element $v\in V$, we say that $v$ is \textbf{tropically
dependent on}  a family $S = \{w_i: i \in I\}$ if $S\cup \{ v \}$
is tropically dependent,
in which case we write $v\hsim S.$ (In particular, $v \hsim \{ v
\}.$)
 A subset $S'$
of $V$ is \textbf{tropically dependent} on $S$ if $v\hsim S $ for
each $v\in S'$.
\end{defn}

 An easy observation:
\begin{rem}\label{change0}  Suppose $S = \{w_i : i \in I\}\subset V$.
 For any given set $\{ \a_i : i \in I\} \subset \tT$ of tangible elements
of~$R$, the set $S$ is tropically  independent iff $\{ \a_i w_i :
i \in I\}$ is tropically  independent. \pSkip

 \end{rem}

 Also recall that any $n+1$ vectors of $\Rz  ^{(n)}$ are tropically
dependent, by \cite[Corollary 6.7]{IzhakianRowen2008Matrices}.


\subsection{Tropical d-bases and rank}\label{dbasea}

\begin{defn}  A \textbf{d-base} (for {\it dependence base}) of a supertropical vector space
$V$ is a
maximal  set of tropically independent elements of $V$.
The \textbf{rank} of a d-base $\tB$, denoted $\rank(\tB)$,  is the
number of elements of $\tB$.
 \end{defn}

Our d-base corresponds to the ``basis'' in
\cite[Definition~5.2.4]{MS09}.

\begin{prop}\label{getbase}Any subspace of $
\Fz  ^{(n)}$  is tropically dependent on any subset $S$ of $n$
tropically independent elements. All d-bases of $\Fz  ^{(n)}$ have
precisely $n$ elements.\end{prop}
\begin{proof}
By \cite[Theorem 6.6]{IzhakianRowen2008Matrices}, the matrix $A$
of $S$ is nonsingular iff $S$  is tropically independent, so in
particular any d-base $\tB$ of $\Fz  ^{(n)}$ must have at least
$n$ elements. On the other hand, any $n+1$ vectors in ~$\Fz
^{(n)}$ are tropically dependent, by \cite[Remark
1.1]{IzhakianRowen2008Matrices}, so $\tB$ has precisely $n$
elements.
\end{proof}

This leads us to the following definition.

\begin{defn} The \textbf{rank}  of a supertropical vector space
$V$ is defined as:
$$ \rank(V):= \max \big \{ \rank(\base) : \base \text{ is a d-base of } V \big \}.$$
\end{defn}

We have just seen that $\rank(\Fz  ^{(n)}) = n.$

\begin{cor}\label{sync} If $V\subset \Fz  ^{(n)}$, then $\rank(V) \le n$.\end{cor}
\begin{proof}
Any d-base of $V$ is contained in a d-base of $\Fz  ^{(n)}$, whose
order must be that of the standard base (to be given in \S
\ref{ssec:4.1}), which is $n$.
\end{proof}

We might have liked $\rank(V)$ to be independent of the choice of
d-base of $V$,
 for any supertropical vector space $V$. This is proved in the
classical theory of vector spaces by showing that dependence is
transitive. However, transitivity fails in the supertropical
theory, since we have the following sort of counterexample.


\begin{example}  In logarithmic notation, over $D(\Real)^{(3)}$, the vector
$v = (0,1,3) $ is tropically dependent on $W = \{ w_1, w_2\},$
where $w_1 = (1,1,2)$ and $w_2 = (1,1,3), $ since $v+w_1+w_2 =
(1^\nu, 1^\nu,3^\nu)$. Furthermore, $W$ is tropically dependent on
$U = \{ u_1, u_2\},$ where $u_1 = (1,1,0)$ and $u_2 = (-\infty,
-\infty,1),$ since $$w_1+u_1+1u_2 = (1^\nu, 1^\nu, 2^\nu), \qquad
w_2+u_1+2u_2 = (1^\nu, 1^\nu, 3^\nu).$$ But $v,u_1,$ and $u_2$ are
tropically independent, since the tropical determinant of the
matrix whose rows are these vectors is $3 \in \tT$.
\end{example}

In fact, different d-bases may contain different numbers of
elements, even when tangible. An example is given in
\cite[Example~5.4.20]{MS09}, which is reproduced here with
different entries.

\begin{example}\label{sum10} Consider the following vectors in
$D(\Real)^{(3)}$:
$$v_1 = (5,5,0),\quad v_2 = (5,5,4),\quad  v_3 = (0,1,4), \quad  v_4 = (0,2,4)  .$$
Then  $v_1,v_2,$ and $v_3$ are tropically dependent (since their
sum $(5^\nu,5^\nu,4^\nu)$ is ghost) and likewise $v_1,v_2,$
and~$v_4$ are tropically dependent. It follows that $\{v_1, v_2\}$
is a d-base for  the supertropical vector space~$V$ spanned by
$v_1, v_2, v_3,$ and $v_4$. But $v_2,v_3,$ and $v_4$ are
tropically independent since their determinant is~$11$, which is
tangible; hence, $\{v_2,v_3,v_4\}$ is also a d-base of $V$.
\end{example}

We do have a consolation.

\begin{lem}\label{tangdep} If the vectors $v_1, \dots, v_{k} \in F  ^{(n)}$ are tropically independent and the vector $v$ is tangible,
then there are $i_1, \dots, i_{k-1}$ in $\{ 1, \dots, k\}$ such
that the vectors $ v_{i_1}, \dots, v_{i_{k-1}}, v $ are tropically
independent.\end{lem}
\begin{proof} Let $A$ be the $k\! +\! 1 \, \times \,  n$ matrix whose rows are
$v_1, \dots, v_{k},v,$ and let $A_0$ denote the  $k\times n$
matrix  of the first $k$ rows $  v_1, \dots, v_{k}.$  By
\cite[Theorem~3.4]{IzhakianRowen2009TropicalRank}, $A_0$ has a
nonsingular $k\times k$ submatrix obtained by deleting $n-k$
columns; deleting these columns in $A$, we have reduced to the
case that $n=k;$ i.e., $A$~is a $k\! +\! 1 \, \times \, k$ matrix.
Now let $A_0' = (a'_{j,i})$ denote the adjoint matrix of $A_0$. We
are done unless for each row $i\le k$, the $k \times k$ submatrix
of $A$ obtained by deleting the $i$ row is singular, which means
that $\sum_{j=1}^k a'_{i,j}a_{k+1,j}$ is ghost. This means that
the vector $(a_{k+1,1},\dots,  a_{k+1,k}) $ g-annihilates the
matrix $A_0'$, which is nonsingular by
\cite[Theorem~4.9]{IzhakianRowen2009Equations}, an impossibility
in view of  \cite[Corollary~6.6]{IzhakianRowen2008Matrices}.
\end{proof}

\begin{prop}\label{thm:tangibleRank} For any tropical subspace $V$ of
$\Fz  ^{(n)}$ and any tangible $v\in V,$ there is a tangible
d-base of $V$ containing $v$ whose rank is that of $V$.
\end{prop}
\begin{proof} Take a tangible d-base of
$V$ of maximal rank, and apply the lemma.
 \end{proof}

 \begin{example} Failure of the analog of Proposition \ref{thm:tangibleRank} for
 non-tangible vectors: Consider the supertropical vector space $W \subset D(\Real)^{(2)}$ spanned by
 $w_1 = (0,1)$ and $w_2 = (0,2)$. Then $v = (1,3^\nu)$  comprises a
 d-base of $W$, consisting of only one element.
 \end{example}

\begin{prop}\label{sumdep1} If $A$ is a matrix of rank $m$, its g-annihilator has a tangible tropically independent set of rank $\ge n-m$.
\end{prop}
\begin{proof} Take $m$ tropically independent rows $v_1, \dots, v_m$ of $A$, which we may
assume are the first $m$ rows of~$A$. For any other row $v_u$ of
$A$ ($m < u \le n$), we have $\bt_{u,1}, \dots, \bt_{u,m}\in \tTz$
such that $v_u + \sum \bt_{i,j} v_i \in   \tGz^{(n)}.$ Letting $B$
be the $(n-m) \, \times n$ matrix whose $(i,j)$ entries are
$\bt_{i,j}$ for $1 \leq i,j \leq m$,  and for which $\bt_{i,j} =
\delta_{i,j}$ (the Kronecker delta) for $m < j \le n$, we see that
$B$ contains an $(n-m) \times  (n-m)$ identity submatrix so has
tangible rank $\ge n-m$, but $BA $ is ghost.
\end{proof}

\begin{example} An example of a $3 \times 3$ matrix $A$ over $D(\Real)$ of rank $m=2$, all of whose entries are tangible,
 although
 $\rank (\Ann (A)) = 2>3-2$.
Take $$A = \left( \begin{matrix} 4 & 4 & 0 \\ 4 & 4 & 1 \\  4 & 4&
2
\end{matrix}\right).$$  $A$ is g-annihilated   by the tropically independent
vectors $v_1 = (1, 1, 0)^\trn$ and $v_2 = (1, 1, 1)^\trn$, since
$Av_1 = Av_2 = (5,5,5)^\trn$.

Note that this kind of example requires $n\ge 3,$ in view of
Theorem~\ref{prdnonsng}.
\end{example}

\subsection{Saturated dependence relations}

 Let
us  study tropical dependence relations in $\Rz^{(n)}$ more
closely.  Example \ref{rmk:bases2}(ii) below shows that   a
tropical dependence of a vector $v$ on an independent set $S= \{
w_i : i \in I\}$ is not determined uniquely. Nevertheless, in this
subsection we do get a ``canonical'' tropical dependence relation,
which we call \textbf{saturated}.
 But first, in order for tropical dependence relations to be
well-defined with respect to the ghost map $\nu: \Rz \to \tGz$, we
verify the following condition.

\begin{lem}\label{compat1} Any submodule of $\Rz  ^{(n)}$
(with the standard ghost submodule $\tHz = \tGz  ^{(n)}$)
satisfies the property that whenever $\a_i, \bt_i \in \tT$ with
$\a_i \nucong \bt_i ,$
\begin{equation}\label{cons} \sum_i \a_i w_i \in \tHz \qquad \text{iff} \qquad \sum _i
\bt_i w_i \in \tHz.\end{equation}
\end{lem}
\begin{proof} The condition clearly passes to submodules, so it is
enough to prove it for $\Rz  ^{(n)}$, and thus, to
check~\eqref{cons} on each component. We write $w_{i,j}$ for the
$j$-component of $w_i$. Note that $\a_i w_{i,j}\nucong  \bt_i
w_{i,j}$ for each $i$. There are two ways for $\sum_i \a_i w_{i,j}
\in \tGz$:
\begin{enumerate}
 \item Some $\a_{i'} w_{i',j}$ dominates $\sum \a_{i} w_{i,j}$
and is ghost, implying $w_{i',j} \in \tGz$, so $$\sum \bt_i
w_{i,j} = \bt_{i'} w_{i',j} = \a_{i'} w_{i',j}  \in \tGz.$$ \item
Two essential summands $\a_{i'} w_{i',j}$ and $\a_{i''} w_{i'',j}$
are $\nu$-matched. But then
$$
\begin{array}{llllll}
  \sum \bt_i w_{i,j}  &  =   \bt_{i'} w_{i',j} + \bt_{i''} w_{i'',j} & =    (\bt_{i'} w_{i',j})^\nu   & &
  \\[2mm]
   & =  (\a_{i'} w_{i',j})^\nu & =   \a_{i'} w_{i',j} + \a_{i''} w_{i'',j}  & =&  \sum \a_i w_{i,j}  \in \tGz.\\
\end{array}
$$
\end{enumerate}
\end{proof}

We examine the tropical dependence
\begin{equation}\label{eq:dependence} v \hsim\sum _{i \in I}\a_i w_i
,
\end{equation}

\begin{lem}\label{sum1} Suppose $V = \Rz  ^{(n)}$. If
$v \hsim \sum _{i \in I}\a_i w_i$ and $v \hsim \sum _{i \in
I}\bt_i w_i,$ for $\a_i,\bt_i \in \tTz$, then taking $\gm _i =
\widehat{\a_i  + \bt_i} ,$ we have
$$v \hsim \sum _{i \in I}\gm _i w_i .$$\end{lem}
\begin{proof} Checking each component in turn, we may assume that $V = \Rz .$
We proceed as
in Lemma~\ref{compat1}. Namely,  $v \hsim \sum _{i \in I}\a_i w_i
$ (resp.~$v \hsim  \sum _{i \in I}\bt_i w_i$) implies one of the
following:
\begin{enumerate}
\item $v$ and some term $\a_{i'} w_{i'}$ dominate (resp.~ $v$ and
$\bt_{i'} w_{i'}$ dominate), in which case $\gm  _{i'} = \a_{i'}$
(resp.~$\gm  _{i'} = \bt_{i'}$). \pSkip

\item $\a_{i'} w_{i'}$  and $\a_{i''} w_{i''}$ dominate,
(resp.~$\bt_{i'} w_{i'}$  and $\bt_{i''} w_{i''}$ dominate), in
which case $\gm  _{i'} = \a_{i'}$ and $\gm  _{i''} = \a_{i''}$
(resp.~$\gm  _{i'} = \bt_{i'}$ and $\gm  _{i''} = \bt_{i''}$).
\pSkip

\item Some ghost term $\a_{i'} w_{i'}$  (resp.  $\bt_{i'} w_{i'}$)
dominates, in which case $\gm  _{i'} = \a_{i'}$ (resp.~$\gm _{i'}
=  \bt_{i'}$).

\end{enumerate}

\end{proof}

Lemma \ref{sum1} gives us a partial order on the coefficients of
the tropical dependence relations of $v$ on a set $S$, and
motivates the following definition:

\begin{definition}\label{def:support}
We say  that the \textbf{support} of a tropical dependence   $\a v
\hsim \sum _{i\in I}\a_i w_i $ (where $\a \in \tT$ and $ \a_i \in
\tTz$) is the set $\{ i\in I : \a_i \ne \rzero\}$. A tropical
dependence of minimal support is called \textbf{irredundant}.
\end{definition}
 \noindent

 A tropical dependence of $v$ on a tropically
independent set $S$ is called \textbf{saturated} if the
coefficients $\al_i$'s in Formula \Ref{eq:dependence} are maximal
possible with respect to $\ge _\nu,$ as defined in
Equation~\eqref{value}; in other words, whenever $ v+ \sum
_{i=1}^\ell \bt_i w_i \in \tGz^{(n)}$ with $\bt_i \in \tTz,$ then
each $\beta_i \le_\nu \a_i .$
%

\begin{rem}\label{sat} If
\begin{equation}\label{dep11} v\hsim \sum_{i=1}^\ell \a_i w_i \end{equation}
is a saturated tropical dependence, then,  for any $k\le \ell$ and
for $v' = v+ \sum_{i=1}^k \a_i w_i$,
\begin{equation}\label{dep12} v'\hsim \sum _{i=k+1}^\ell \a_i w_i   \end{equation}  is also a saturated tropical
dependence, since any $\nu$-larger tropical dependence for
\eqref{dep12}  would yield the corresponding $\nu$-larger tropical
dependence for \eqref{dep11}.
\end{rem}

\begin{thm}\label{sat1} Suppose $V = \Fz  ^{(n)},$ for a supertropical semifield $F = (F, \tG, \nu)$.
Any irredundant tropical dependence
\begin{equation}\label{dep1} v\hsim \sum _{i=1}^\ell \a_i w_i  \end{equation}  can be increased to a unique (up to
equivalence in the sense of Remark~\ref{oneone}) saturated
tropical dependence of $v$ on $S = \{ w_1, \dots, w_\ell\}$,
having the same support.
\end{thm}

\begin{rem}\label{sat10} When the vector  $v$ is tangible, and $S$ is a d-base, Theorem \ref{sat1}
is an immediate consequence of \cite[Theorems~3.5
and~3.8]{IzhakianRowen2009Equations}, which shows that $A x \lmodg
v $ has the maximal tangible vector solution $x = \hnu (A^\nb v)$
(where $A^\nb = \frac 1{|A|}\adj{A}$),  which in view of
Lemma~\ref{surpass2}   is also a solution for  $ Ax \hsim v$. Here
we take $A$ to be the matrix of $ S,$ which is nonsingular,  and
$x$ to be the vector $(\al_1, \dots, \al_\ell)^\trn$.

In general, $x = A^\nb v$ is a solution for the matrix equation $
A x \lmodg v,$ which, when $v$ is written as a row,   is $x
A^{\operatorname t} \lmodg v.$ (In a sense, row form is more
natural, since the matrix of $S$ is obtained from the rows.) But
this vector need not be tangible.
\end{rem}

Here is a  direct combinatoric proof of Theorem~\ref{sat1} that
does not rely on matrix theory, and does not depend on the
additional assumption of tangibility of $ S $.
\begin{proof}  Uniqueness of a saturated tangible solution is obvious, since one could just take the sup of
any two distinct saturated tropical dependences to get a
contradiction. This also gives the motivation for proving
existence. Write $v = (v_1, \dots, v_n).$ We start with some
tropical dependence \eqref{dep1}, which need not be saturated,
with the aim of checking whether we can modify it until it is
saturated. In principle, we could increase the $\nu$-values of the
coefficient $\a _i$  if at each component $j$ of the vector $\a _i
w_i$ the $\nu$-value of $v_j$ is not attained, and this is the
main idea behind the proof. But increasing $\a_i$ still may not
yield a saturated tropical dependence, since the coefficient may
be allowed to increase further, so long as some other term in the
tropical dependence also is adjusted so as to have a $j$-component
of the same $\nu$-value. Since these $j$-components are the most
difficult to keep track of,  we  pay special attention to them.
Write $w_{i,j}$ for the $j$-component of $w_i$.

We say that an index $j \le n$  has
\textbf{type 1} if $v_j$ is not dominated by $\sum \a_j w_{i,j},$ which means that either $v_j$ itself is ghost, or
else precisely one $w_i$ has $\a_iw_{i,j}$ matching $v_j$ and this
$w_{i,j}\in \tT$.

We say that $j$  has \textbf{type 2} for $v$ if $v_j$ is
 dominated by $\sum \a_j w_{i,j},$ which means that either
there exists~$i$ such that  $w_{i,j}$ is ghost and dominates $v_j$
or there are
   $ i,  i'$ such that both $ \a_iw_{i,j}$  and
$ \a_{i'}w_{i',j}$ dominate~$v_j$.

Note that increasing the coefficients $\a_i$ in a tropical
dependence cannot change the type of an index~$j$ from type 2 to
type 1. Also, at least one index must have type 1, since otherwise
$\sum \a_iw_{i,j}\in \tGz^{(n)},$ contrary to the hypothesis that
the $w_i$ are tropically independent. We choose our tropical
dependence such that the number of indices of type 1 is minimal.
In this case, if $ \a_iw_{i,j}$ $\nu$-matches $v_j$ for $j$ of
type~1, we cannot find a $\nu$-greater tropical dependence in
which $\a_i$ is increased, since this would force the tropical
dependence to have an extra type 2 index. Thus, in this case we
say $w_i$ is \textbf{anchored} at~ $j$. Renumbering the vectors,
we may assume that $w_1, \dots, w_k$ are anchored at various
indices, and replace $v$ by $v' = v+ \sum_{i=1}^k \a_i w_i.$ Now
we have a new tropical dependence $v' + \sum _{i=k+1}^\ell \a_i
w_i \in \tGz^{(n)},$ which by induction on~$\ell$ can be increased
to a saturated tropical dependence
$$v' \hsim  \sum _{i=k+1}^\ell \a_i' w_i  .$$ But then the
tropical  dependence
$$v \hsim \sum _{i= 1}^k
\a_i w_i +\sum _{i=k+1}^\ell \a_i' w_i $$ is saturated, since
$w_1, \dots, w_k$ are anchored.
\end{proof}

\begin{prop}\label{sumdep} If \begin{equation}\label{dep33} v\hsim  \sum _{i=1}^\ell \a_i w_i , \qquad v'\hsim  \sum _{i=1}^\ell \a'_i w_i \end{equation} are saturated tropical  dependences, then
\begin{equation}\label{dep34} v+ v'\hsim  \sum _{i=1}^\ell \widehat{(\a_i
+\a'_i)}  w_i \end{equation} also is a saturated tropical
dependence.
\end{prop}
\begin{proof} Again we have two proofs, the first using results from \cite{IzhakianRowen2009Equations} in the case when $v,v'$ are
tangible
 and the
matrix $A$ of the $w_i$ is nonsingular. In the first case, one just takes the
solutions $x = \Inu{A^\nb v}$ and $x' = \Inu{A^\nb v'}$ for the
vectors of the $\a_i$ and the $\a_i'$, and then note that $$\hnu
(A^\nb v+ A^\nb v')= \hnu (A^\nb (v+   v')).
$$

For the general case, one needs to modify the second proof of
Theorem~\ref{sat1} for the vector $v+v'$. Namely, consider the
tropical dependence $$ v+ v'\hsim  \sum _{i=1}^\ell \gm _i w_i
,$$ where $\gm _i= \widehat{(\a_i +\a'_i)}.$ At least one index in
this tropical dependence must have type 1 for $v+v'$, since
otherwise  the $w_i$ are tropically dependent. We choose our
tropical dependence such that the number of indices of type 1 is
minimal. As before, if $ \gm _i w_{i,j}$ $\nu$-matches $v_j$ for
$j$ of type 1 we cannot find a larger tropical dependence in which
$\gm _i$ is increased, so $w_i$ is anchored at $j$. Again, we may
assume that $w_1, \dots, w_k$ are anchored at various indices, and
replace $v+v'$ by $$v'' = v+v'+ \sum_{i=1}^k \gm _i w_i .$$ But
$$ v+ \sum_{i=1}^k \a_i w_i \hsim \sum _{i=k+1}^\ell \a_i w_i   \qquad \text{and} \qquad  v'+ \sum_{i=1}^k \a_i'
w_i\hsim \sum _{i=k+1}^\ell \a_i' w_i $$ are saturated tropical
dependences by Remark~\ref{sat}, so, by induction on~$\ell$,
$$v'' \hsim  \sum _{i=k+1}^\ell \gm _i w_i  $$   is a
saturated tropical dependence. But then the tropical dependence
$$v \hsim  \sum _{i= 1}^k
\gm _i w_i +\sum _{i=k+1}^\ell \gm _i w_i  $$ is saturated.
\end{proof}

\section{Tropical spanning}

In this section, we continue to consider the fundamental question
of what ``base'' should mean for supertropical vector spaces. The
d-base (defined above) competes another notion to be obtained from
$\lmodh$. But at the moment we turn to the naive analog from the
classical theory of linear algebra.

\subsection{Classical bases}\label{ssec:4.1}
\begin{defn}\label{def:classicalBase} A module $V$ over a semiring~$R$ is \textbf{classically spanned}
by a set  $S = \{ w_i: i \in I \}$ if~every element of~$V$ can be
written in the form $$v=\sum_{i \in J} r_i w_i,$$ for $r_i \in R$
and some finite index set $J \subset I$.

A set $\base = \{b_1, \dots, b_n\}\subset V$ is a
\textbf{classical base} of a module $V$ over a semiring $R$, if
every element of~$V$ can be written uniquely in the form $\sum_{i
=1}^n r_i b_i$, for $r_i \in R.$
In this case, we say that $V$ is \textbf{classically free} of
\textbf{rank} $n$.
\end{defn}

For example, the \textbf{standard base} of $\Rz  ^{(n)}$ is the
classical base defined as
\begin{equation}\label{eq:standardBase}
\varepsilon_1 = (\rone,\rzero, \dots, \rzero), \quad \varepsilon_2
= (\rzero,\rone,\rzero, \dots, \rzero), \quad \dots, \quad
\varepsilon_n = (\rzero,\rzero, \dots, \rone).
\end{equation}

\begin{prop} If $V$ is classically free of rank $n$, then $V$ is
isomorphic to $\Rz  ^{(n)}$. \end{prop}

The proof is standard; taking a classical  base $b_1, \dots, b_n$,
one defines the isomorphism $\Rz  ^{(n)}\to V$ by $$(r_1, \dots,
r_n) \mapsto \sum _{j=1}^n r_j b_j.$$

\subsection{Tropical spanning}


 \begin{defn} A vector $v\in V$ is \textbf{tropically spanned}
 by a  set $S = \{ w_i : i \in I\} \subset V$ if there exists a nonempty finite
subset $I'\subset I$ and a family $\{\a_i : i \in {I'}\}\subset
\tT$, such that

\begin{equation}\label{eq:dependence2} v \lmodh \sum _{i \in I'}\a_i
w_i.
\end{equation}
In this case, we write $v\lmodh S.$

A subset $S' \subseteq V$ is \textbf{tropically spanned}
 by  $S$, written $S' \lmodh S,$ if $v\lmodh S$ for each $v\in S'$.
\end{defn}

\begin{rem}[Transitivity  for tropically spanning]\label{transit}
 If $V\lmodh W$ and $W\lmodh U$, then $V\lmodh U$.\end{rem}

Obviously, any set classically spanned by $S$ is tropically spanned; surprisingly, the converse often holds.

\begin{rem}\label{rmk:span} $ $
\begin{enumerate} \eroman
  \item
Any   element  tropically spanned by $S = \{w_i: i \in I\}$ is
  tropically dependent on $S$. \pSkip

   \item If an \atangible\ vector $v\in V$ is
  tropically spanned by a set $S\subset V$, then $v$ is classically spanned by
  ~$S$. \pSkip

  \item  The assertion (ii) can fail for nontangible $v\in \Rz  ^{(n)}$; take $S = \{(\rone, \rone)\} \subset \Rz ^{(2)}$,
viewed as an $R$-module, then $(\rone, \rone^\nu)$ is tropically
spanned by~$S$, but not classically spanned by $S$. \pSkip

\item  If $V$ has a classical spanning set $\mathcal B$ of
\atangible\ vectors, and $\mathcal B$ is tropically spanned by a
set $S$, then $V$ is classically
  spanned by $S$, by (ii) and
  transitivity.
  In particular, if $\Rz  ^{(n)}$ is tropically spanned by a set $S$, then $\Rz
^{(n)}$ is classically
  spanned by $S$, since $\Rz  ^{(n)}$ has the standard base. \pSkip

\item Any element tropically spanned by $S$ is also tropically
    dependent on $S$, but not conversely; for example $v= (\rone,\rone) \in
    R^{(2)}$
    is tropically dependent on $S = \{ (\rone,\rone^\nu) \} \subset R^{(2)}$,
    viewed as $R$-module, but $v$ is not tropically spanned by $S$. This leads to an interesting dichotomy to be studied
shortly.\pSkip

\end{enumerate}
\end{rem}

Thus, we see that \atangible\ vectors already begin to play a
special role in the theory of tropical dependence.

\begin{remark}\label{rmk:tnrSpan} Tropical spanning does not satisfy the assertion analogous
to  Lemma \ref{sum1}. For example, take $$ \{ w_1 = (1,2), w_2 =
(1,3) \} \subset D(\Real)^{(2)}$$ and the vector $v = (1,3^\nu)$,
then $v \lmodh w_1$ and $v  \lmodh  w_2$, but $v \nlmodh  w_1 +
w_2 = (1^\nu, 3)$.
\end{remark}

\begin{lem}\label{lem:span} $W = \{ v\in V: v\lmodh S\}$ is a subspace of
$V$ for any $S \subset V$.
\end{lem}
\begin{proof} If $v = \sum _{i \in I}\a_i
w_i +y$ and $v' = \sum _{i \in I}\a_i ' w_i+ z,$ where $\a_i,
\a'_i \in \tT, $ $w_i \in S$ and $y, z \in \tHz$, then letting $J
= \{ i: \a_i \nucong {\a_i'} \},$ we have, by bipotence,
 $$v+v' = \sum
_{i \notin J}   \beta_i w_i +\sum _{i \in J} \a_i ^\nu
 w_i+ (y + z) \lmodh  \sum
_{i \notin J} \beta_iw_i,$$ where $\beta_i \in \{ \a_i
,\a_i'\}\subset \tT.$
 The other verifications are easier.
\end{proof}

We call $W$ (in Lemma \ref{lem:span}) the subspace
\textbf{tropically spanned} by $S$, and say that $S$ is a
\textbf{tropically spanning set} of $W$.

A supertropical vector space is \textbf{finitely spanned} if it
has a finite tropically spanning set.

%

\begin{example}\label{rmk:bases2} Take $R = D(\Real)$, with logarithmic
notation. 
\begin{enumerate} \eroman
\item  The vectors
$$v_1 =(1,0,1), \quad v_2 = (1,1,0), \quad \text{
and }  \ \ v_3 =(0,1,1)$$ are tropically dependent in $
D(\Real)^{(3)}$, since their sum is $(1^\nu,1^\nu,1^\nu)$. None of
these vectors is tropically spanned by the two other vectors.
\pSkip

\item Even when a vector is classically spanned by tropically
independent vectors, the coefficients need not be unique. For
example, $$(4,5) = 2(1,1) + 2(2,3) = 1(1,1) + 2(2,3).$$ The point
of this example is that the first coefficient is sufficiently
small so as not to affect the outcome. \pSkip

\item  Another such example: The vectors
$$v_1 =(-\infty,-\infty,1), \quad v_2 = (1,1,-\infty), \quad \text{
and }  \ \ v_3 =(-\infty,1,1)$$ are tropically independent,
although classical spanning with respect to them (and thus also
tropical spanning) is not unique; e.g., $(3,3,1) = 2v_2 + v_3 =
v_1 + 2v_2$. \pSkip

\item  Another such example: Consider the vectors
$$v_1 =(1,4,3), \quad v_2 = (2,3,4), \quad \text{
and }  \ \ v_3 =(0,20,20).$$ Then $(3,20,20) = 1v_2 + v_3 = 3v_1 +
v_3$.

\item  Another such example: Consider the space $V$ spanned by the
five critical vectors
$$(0,-\infty,0,-\infty,0,-\infty),\quad
 (-\infty,0,-\infty,0,-\infty,0),$$
 $$(0,-\infty,-\infty,0,-\infty,-\infty), \quad
 (-\infty,0,-\infty,-\infty,0,-\infty), \quad
(-\infty,-\infty,0,-\infty,-\infty,0).$$ Then $(0,0,0,0,0,0)$ is
the sum of the first two vectors as well as the last three.

\end{enumerate}

\end{example}

It does not follow from Lemma~\ref{sumofsat} that for $S = \{w_1,
\dots, w_n\}$, there is a $\nu$-maximal set of  $\a_1, \dots,
\a_\ell \in~\tT$ such that $v \lmodh  \sum _{i=1}^\ell \a_i w_i $.
For example, in logarithmic notation take $$v = (1,1), \quad  w_1
= (1,0), \quad \text{and } \ \  w_2 = (1,1).$$ Then $v = \a w_1 +
w_2$ for all $\a < 0,$ but taking $\a = 0$ yields $w_1+w_2 =
(1^\nu,1) $. 

%
%
%

\begin{prop}\label{sd} For any subspace $V$ of $F^{(n)}$, the number of elements of
any tropically spanning set~$\gen$  of $V$ is at least $\rank(V)$.
\end{prop}
\begin{proof} Take a d-base $\{ v_1, \dots, v_m \}$ of $V,$ where
$m = \rank(V)\le n.$  By \cite[Theorem
3.4]{IzhakianRowen2009TropicalRank}, the $m\times n$ matrix  whose
rows are $v_1, \dots, v_m$ has rank $m$. Taking a nonsingular
$m\times m$ submatrix and erasing all the $n-m$ columns  not
appearing in this submatrix, we may assume that $m=n$ (since we
still have a supertropically generating set which we can shrink to
a minimal one).

Writing $v_i \lmodg \sum \a_{i,j} s_j$ for suitable $s_j \in
\gen,$ we see that some matrix whose rows are various $s_j$ is
nonsingular, implying that some subset of $m$ vectors of $\gen$ is
tropically independent, and thus $|\gen | \ge m.$
\end{proof}

\subsection{s-bases}

We are ready for another version of base.

\begin{defn}\label{def:genreating} An \textbf{s-base} (for {\it supertropical base})
of a supertropical vector space $V$ (over a supertropical
semifield $F$) is a minimal tropical spanning set $\gen$, in the
sense that no proper subset of $\gen$ tropically spans $V$.
\end{defn}

As we shall see in Example \ref{onemoreex} below,    a vector
space with a finite d-base could still fail to have an s-base.
Even when an s-base exists, it could be considerably larger than
any d-base.

\begin{example}\label{exmp:strip1} Elements of a vector space $V$ may be tropically
dependent on a subspace $W$ but  not tropically spanned by $W$, as
indicated in Example \ref{rmk:bases2}(i).
\end{example}

\begin{example}\label{rmk:cbase1} Let V be the subspace of
$\Rz^2$ spanned   by $S = \{(1,1),(1^\nu,1), (1,1^\nu)\}$ in
logarithmic  notation, equipped with the standard ghost module.

Each of these vectors alone comprises a d-base of V, whereas $S$
is an s-base of V.
\end{example}


Note that an s-base $\gen$ need not be finite. On the other hand,
obviously any finite tropical  spanning set contains an s-base, so
any finitely spanned vector space has an s-base. In order to
coordinate the definitions of s-base and d-base we introduce the
following definition.

\begin{defn}\label{def:genreating1} A \textbf{d,s-base} is an s-base which is also a d-base.
A supertropical vector space $V$ is \textbf{finite dimensional} if
it has a finite d,s-base.
\end{defn}

\begin{prop} The cardinality of the s-base $\gen$ of a finite dimensional vector space
$V$ is precisely  $\rank(V)$.\end{prop}
\begin{proof} $|\gen| \ge \rank(V)$ by Proposition~\ref{sd}. But
we get equality, since by definition $\gen$ is itself a d-base.
\end{proof}

\begin{example} Suppose $S$ is a tropically
independent subset of $V$. Then $S$ is a d,s-base of the subspace
of $V$ tropically spanned by $S$. These are the subspaces of
greatest interest to us, and will be studied further, following
Definition~\ref{sdbase0}.
\end{example}

\begin{example} There are four possible
sorts of nonzero subspaces of $\Fz^{(2)}$ tropically spanned by a
set~$S$ of tangible elements over a supertropical semifield $F$,
writing $\{\varepsilon_1 = (\rone, \rzero),\ \varepsilon_2 =
(\rzero, \rone)\}$ for   the standard base: \pSkip
\begin{enumerate} \eroman

\item   The  plane $
\Fz^{(2)}$ itself.\pSkip
\item A half-plane -- of tangible rank 2, having tangible s-base
containing $\varepsilon_1$ or $\varepsilon_2$, as well as  one
tangible element $ \a _1 \varepsilon_1 + \a_2 \varepsilon_2$ for
$\a_1, \a_2 \in \tT$; \pSkip
\item  A planar strip  -- of tangible rank 2, having tangible
s-base $\{ \a _1 \varepsilon_1 + \a_2 \varepsilon_2, \bt _1
\varepsilon_1 + \bt_2 \varepsilon_2\},$ where  $\a_1, \a_2, \bt
_1, \bt _2 \in \tT$; \pSkip
\item A  subspace of tangible rank 1, each pair of whose elements
are tropically dependent. The  tangible vectors are all multiples
of a single vector.  \end{enumerate}

One also has  examples of non-tangibly generated subspaces of
$\Fz^{(2)}$, such as $W = \{ (\a, \a^\nu): \a \in F\}$.
\end{example}

\subsection{Critical elements versus s-bases}

 Since s-bases are
involved in the actual generation of the space, they are more in
tune with the classical theory of convexity, and can be studied
combinatorially. Here is another way to view the s-base, which is
inspired by the  literature on convex spaces. For convenience, we
take $R$ to be a supertropical semifield. We say that two elements
$v, w$ in a supertropical vector space $ V$ are
\textbf{projectively equivalent}, written $ v \sim w$, iff $ v =
\al w$ for some tangible element $\al \in R$. Accordingly, we
define the equivalence class of $v$ as
$$ [v]_\sim := \{ w \in V \ | \ w \sim v\}.$$

\begin{defn}\label{def:critical}
A vector $v$ in a supertropical vector space $V$ is
\textbf{critical} if we cannot write $v \lmodg v_1 + v _2$ for
$v_1,v_2 \in V \setminus [v]_\sim.$ Taking one representative for
each class $[v]_\sim$, a \textbf{tropical critical set} of $V$ is
defined as a set of representatives of all the critical elements
of $V$.
\end{defn}

Critical elements correspond to ``extreme points'' over the
max-plus algebra in \cite{GK}, who show that every point in
$R^{(n)}$ is a linear combination of at most $n+1$ extreme points.
There is a basic connection between criticality and \atangible.

\begin{lem}\label{gsprop} Suppose $v \lmodh \a v +w$ for $\a \in \tT$, $v,w \in V,$ and $v\notin \tHz.$
Then $\a \le_\nu  e$. Furthermore:

\begin{enumerate}
\item If  $\a <_\nu e$, then $v \lmodh w.$ \pSkip

\item Suppose $\a \in \tT_e$; i.e., $\a \nucong  e$. If $w \in
\tHz,$ then $v = \a   v$. For any $w \in V,$   $$v = \a ^2 v +ew'=
\a ^2 v ,$$ where $w' \lmodh w.$ \end{enumerate}
\end{lem}
\begin{proof} Write $v = \a v +w',$ where $w' \lmodh w.$

(1):  If $\a >_\nu  e,$ then $$
\begin{array}{lll}
v  & =  \a v +w'  =  (\a+\fone )v +w' & \\ &  = v + \a v +w'
 = v+v = ev\in \tHz,  \end{array}$$ a contradiction. Hence, $\a \le_\nu  e$.

  If $ \a  <_\nu e$, then $ \a = \a+\fone$, implying $$v = (
\a+\fone )v = \a v +v =  \a v+ \a v+w' = e \a v +w' \lmodh w,$$
proving (1).

(2): Thus, we assume that $ \a \in \tT_e$. If   $w = ew,$  then
$$
\begin{array}{lll}
 v  & = \a v + w' = \a (\a v + w') +w'  &  \\
   & = \a ^2 v + (\a  + \fone ) w' = \a
^2 v + e w' \\ & = \a (\a v + ew') = \a v. &  \\
\end{array}
 $$ For any $w,$ if $\a \in
\tT_e$, then  $$v = \a v + w' = \a (\a v + w') +w' = \a ^2 v + (\a
+ \fone ) w' = \a ^2 v + e w'.$$ Hence, $v = \a ^2 v$ by the
previous assertion.
\end{proof}

\begin{prop} Any critical element $v\in V$ is \atangible.
\end{prop}
\begin{proof} Otherwise $v = w + w'$ for suitable $w  \in V$,
$w' \in \tHz,$ for which $w \ne v,$ but by criticality, $w = \a v$
for $\a \in \tT.$ First assume that $v\notin \tHz.$ Then, by
Lemma~\ref{gsprop}, $\a \le_\nu  e,$ and furthermore $\a \in
\tT_e,$ since otherwise $v \lmodg ew'$ contrary to $v\notin \tHz.$
But now, by  Lemma~\ref{gsprop}, $v = \a  v = w,$ a contradiction.

Hence we may assume that $v \in \tHz,$ and thus $$w = \a v = (\a
e) v = e v = v,$$ again a contradiction.
\end{proof}

\begin{lem}\label{enoughspan}
 An \atangible\ element $v\in V$ is  critical iff it is not tropically
 spanned by $V \setminus [ v ]_\sim $, i.e. $v
\nlmodh \sum \al_i w_i$ for any $\al_i \in \tT$, $w_i \in V
\setminus  v $.
\end{lem}
\begin{proof} $(\Rightarrow)$ Suppose on the contrary that $v \lmodh \sum _{i=1}^t \al_i
w_i$; by definition of criticality, $t>1$. Then taking $v_1 = \a_1
w_1$ and $v_2 = \sum _{i=2}^t \al_i w_i$, we must have $v_2 \in
[v]_\sim,$ and conclude by induction on $t.$
\end{proof}


 Clearly a tropical critical set of a vector space $V$ is  projectively  unique, but could be empty.

\begin{example}\label{onemoreex} $ $
\begin{enumerate} \eroman
    \item The standard base $ \{ \varepsilon_1, \dots, \varepsilon_n\}$ of $\Rz  ^{(n)}$ is also its
tropical critical set. \pSkip

    \item The  tropical critical set of the subspace $W = \Rz^{(2)} \setminus ( [\varepsilon_1]_\sim \cup
    [\varepsilon_2]_\sim)$ is empty. \pSkip

 \item  $W= \Rz^{(2)}\setminus [\varepsilon_1]_\sim$ has the  tropical critical set $ [\varepsilon_2]_\sim$,
but has no s-base.

\end{enumerate}

\end{example}

Despite the last two examples, some positive information is
available.
\begin{lem}\label{enoughspan1}
 Any tropical spanning  set $\gen$ contains a  tropical  critical set of $ V$.
\end{lem}
\begin{proof}  Suppose $v \in V$ is critical.   By
hypothesis on $\gen$, $v$ is tropically spanned by $\gen$ but, by
Lemma~\ref{enoughspan}, it must be an element of $\gen$ (up to
projective equivalence).
\end{proof}

\begin{thm}\label{thm:uniqueGenSet} Suppose $V$ has an s-base $\gen$.
Then $\gen$ is precisely the  tropical  critical set of $ V$.
\end{thm}

\begin{proof}
In view of Lemma~\ref{enoughspan1}, it remains to show that each
element of  $\gen$ is critical. Suppose $v \in \gen$ is not
critical. Then $v = v_1 + v_2$ where $v_1, v_2 \notin \tT v.$
Thus, when we write $$v_1 = \sum \a _{1,i} s_{1,i} + w_1 \quad
\text{ and } \quad v_2 = \sum \a _{2,i} s_{2,i} + w_2$$ for $\a
_{1,i},\a _{2,i}\in \tT$ and $w_1,w_2 \in \tHz,$ we must have $v$
appearing in one of the sums (for otherwise $v=  v_1 + v_2$ is
tropically spanned by the other elements of $\gen$, contrary to
hypothesis).

Thus, we may assume $s_{1,1} = v,$ and  we have
 $$v_1 \lmodh \a_1 v + \sum_{i\ne 1} \a _{1,i} s_{1,i}  $$ and
 similarly  $v_2\lmodh \a_2 v + \sum_{i\ne 1} \a _{2,i} s_{2,i}  $.
 (Formally, we permit $\a_2 = \fzero.$)  We also write $v_j = \a_j + x_j$ where
   $x_j\lmodh \sum_{i\ne 1} \a_{j,i} s_{j,i}.$

 Now $$v = v_1+v_2 =
\beta v + x,$$ where $\beta = \a _{1,1} + \a _{2,1}$ and $x =
x_1+x_2.$ But then $\beta \le_\nu  e,$ by Lemma~\ref{gsprop},
which also says that if $\beta <_\nu e,$ then $v \lmodg x,$
contrary to $S$ being an s-base. Thus, we may conclude that $\beta
\nucong  e.$ By symmetry, we assume that $\a_1 \nucong e.$ If
$\a_2 <_\nu e,$ then $v_2 \lmodh x_2,$ and
$$
\begin{array}{lll}
v & = v_1 + v_2 = \a_1v + x_1 +v_2 \\ & = \a_1(\a_1v + x_1 +v_2) +
x_1 +v_2 \\ & = \a_1^2 v + (\a_1+\fone ) (x_1+v_2) \\ &  = \a_1^2
v + e( x_1 +v_2),\end{array}
$$
and thus $$
\begin{array}{lll}
v _1 & = \a_1^2 v + e( x_1 +v_2) \\ & = \a_1(\a_1^2 v + e( x_1
+v_2))+  x_1 +v_2  \\ & = \a_1^3v +e( x_1 +v_2) \\ & = \a_1(\a_1^2
v +e( x_1 +v_2))= \a_1 v.\end{array}  $$ Thus, we are done for
$\a_2 <_\nu e,$ and may assume that $\a _2 \in \tT_e.$ Then $$v =
(\a_1 + \a_2) v + x_1 +x_2 = ev +x,$$ implying $v = ex \in \tH$
and thus $\a_j v = \a_j e v = ev$ for $j = 1,2.$ Hence
$$v = v_1 + v_2 = \a_1v + x_1 +\a_2v + x_2 = (\a_1 +  \a_2)v +x =
ev +x,$$ and thus $v = ev + ex $ by Lemma~\ref{gsprop}, implying
$$
\begin{array}{lll}
v_1  & = ev + x_1 = ev+ex + x_1 \\ & = ev + ex_1 + ex_2 + x_1 \\
&  = ev + ex_1 + ex_2 = ev + ex =v.\end{array}
$$
\end{proof}

This theorem is generalized in
\cite{IzhakianKnebuschRowen2010Modules}.

\begin{cor}\label{cor:uniqueGenSet}
The s-base  (if it exists)  of a supertropical vector space is
  unique  up to multiplication by
tangible elements of $R$, and is
 comprised of \atangible\ elements.
\end{cor}


By Corollary \ref{cor:uniqueGenSet}, we have the following
striking result:
\begin{thm}\label{cor:uniquebase} The s-base (if it exists) of a supertropical vector  space is
 unique up to multiplication by scalars.
\end{thm}

\begin{example}\label{rmk:cbase} The only s-bases of the supertropical vector space
 $V = \Rz  ^{(n)}$ are its
classical bases $S = \{ \a_1\varepsilon_1, \dots, \a_n
\varepsilon_n\}$, where $\al_1,\dots,\al_n \in \tT$.
\end{example}

One also has the following tie between critical sets and s-bases.

\begin{prop}\label{thm:uniqueGenSet1} 
Any critical set  $\crit$ of  a supertropical vector space $V$ is
an s-base  of the subspace $W$ tropically spanned by $\crit$.
\end{prop}\begin{proof} By hypothesis, $\crit$ tropically  spans $W$,
so we need only check minimality. But for any $v \in \crit,$ by
definition, $\crit \setminus \{v\}$ does not tropically span $v$.
\end{proof}

$ $


\subsection{Thick subspaces}
\begin{defn}\label{def:thick}
A subspace $W$ of a supertropical vector space $(V, \tHz)$ is
\textbf{thick} if $\rank(W)= \rank(V).$
\end{defn}
For example,  the subspace $  \a V\subseteq V$ is  thick, for any
$\a \in \tT.$ Likewise, any subspace of $\Rz  ^{(n)}$ containing
$n$~tropically independent vectors is thick.

\begin{rem}\label{def:thick0} By definition, any thick subspace of
a thick subspace of $V$ is thick in $V$. \end{rem}

\begin{rem}\label{def:thick1}
Any thick subspace $W$ of a supertropical vector space  $(V,
\tHz)$ contains a d-base of~ $V$. Indeed, by definition, for $n =
\rank(V),$ $W$ contains a set of $n$ tropically independent
elements, which must be a maximal tropically independent set in
$V$, by definition of rank.

Thus, $V$ is tropically dependent on any thick subspace.
\end{rem}

\begin{example}
There exists an infinite chain of thick subspaces  $W_1
 \subset W_2 \subset \cdots$ of  $V= D(\Real)^{(2)}$,  where $W_k$ is the strip tropically spanned by $\{ (k,0), (0,k) \}$, $k \in
 \Net^+$. Thus, $\{ (k,0), (0,k) \}$ is
 not an s-base of
$D(\Real)^{(2)}$. (One
 could expand this to an uncountable chain by taking $k \in
 \Real^+$.)
\end{example}

\subsection {Change of base matrices}

 We write $P_\pi$ for the permutation matrix whose entry in
the $(i, \pi (i))$ position is $\rone$ (for each $1\le i \le n$)
and $\rzero$ elsewhere. Likewise, we write $\diag\{a_1, \dots,
a_n\}$ for the diagonal matrix whose entry in the $(i,i)$ position
is $a_i$ and $\rzero$ elsewhere, and denote it as $D$.  We call
the product $P_\pi D$ of a permutation matrix and a tangible
(nonsingular) diagonal matrix, with each diagonal entry $\neq
\rzero$, a \textbf{\dper}\ matrix,  and denote it as
$\perm_{\pi;D}$.

Recall from \cite[Proposition 3.9]{IzhakianRowen2008Matrices} that
over a supertropical semifield, a matrix is invertible iff it is a
\dper\ matrix $\perm_{\pi;D}$ with $D$ nonsingular. In particular,
the set of all \dper \ matrices form a group whose unit element is
$I$.

\begin{defn} Given an s-base  $\base = \{v_1, \dots, v_n\}$ and
another s-base $\base' = \{v_1', \dots, v_n'\}$ of $V \subseteq
\Fz  ^{(n)}$, whose respective row matrices   are denoted $A$ and
$A'$, a \textbf{change of base matrix} is a matrix $P$ such that
\begin{equation}\label{eq:changeBase} A' = P A ;
\end{equation}
\end{defn}

\begin{prop}\label{prop:changebase1} The generalized permutation
matrices are the only change of base matrices of s-bases (and thus
classical bases).
\end{prop}

\begin{proof}
Immediate by Theorem \ref{cor:uniquebase}.
\end{proof}

\begin{rem}\label{changebase2} It follows from Proposition
\ref{prop:changebase1}, applied to the standard base, that the
matrix $A$ is the matrix of a classical base iff $A$ is a \dper\
matrix.
\end{rem}

\begin{example} Any classical base of $\Rz  ^{(n)}$ (after reordering indices) must
be of the form
 $$
 b_1 = ( r_1,\rzero, \dots, \rzero), \quad b_2 =
(\rzero, r_2 ,\rzero, \dots, \rzero), \quad \dots, \quad b_n =
(\rzero,\rzero, \dots, r_n),  $$ where $r_i \in \tT$ are invertible
and tangible.
\end{example}

\section{Linear transformations of supertropical vector spaces, and the dual space}

Our main goal in this section is to introduce supertropical linear
transformations, and use these to define the dual space with
respect to a d,s-base $\tB$, and to show that it has the canonical
dual s-base given in Theorem~\ref{dualbasethm}; this enables us to
identify the double dual space with $V_\tB$. (A version of a dual
space for idempotent semimodules, in the sense of dual pairs,
leading to a Hahn-Banach type-theorem is given in \cite{CGQ}.)

\subsection{Supertropical maps}

Recall that a \textbf{module homomorphism} $\varphi:V\to V' $ of
modules over a semiring~$R$ satisfies
$$\varphi(v+w) = \varphi(v) + \varphi(w),\qquad \varphi(av) = a\varphi(v), \qquad \forall a \in R, \ v,w \in V.$$

We weaken this a bit in the supertropical theory.

\begin{defn} Given supertropical vector spaces $(V, \tHz)$ and
$(V', \tHz ')$ over a supertropical semifield~$F$, a
\textbf{supertropical map}
$$\tmap : (V,  \tHz)  \ \to \ (V', \tHz ')$$ is a function satisfying
\begin{equation}\label{tamp0}\tmap (v+w)
\lmodh \tmap (v) + \tmap (w),\qquad \tmap (\a v) = \a\tmap (v),
\qquad \forall \a \in F, \ v,w \in V,\end{equation} as well as
\begin{equation}\label{tamp1} \tmap(\tHz) \subseteq
\tHz'.\end{equation}

We write $ \Hom (V, V' )$ for the set of supertropical maps from
$V$ to $V',$ which is viewed as a vector space over $F$ in the
usual way. A supertropical map is \textbf{strict} if it is a
module homomorphism.
 \end{defn}

 The modules over a given semiring with ghosts   form a
category, whose morphisms are the supertropical maps of  modules
with ghosts.

\begin{rem}  The second  condition of \eqref{tamp0} implies
$$\tmap(v^{\nu}) = \tmap(e v) = e \tmap(v) = \tmap(v)^{\nu
'};$$ i.e., $ \tmap  \circ \nu = \nu' \circ \tmap.$

 In case $\tHz =e V,$ the standard ghost submodule,  \eqref{tamp1} follows formally from  \eqref{tamp0}. \end{rem}

\begin{rem} One may wonder why we have required  $\tmap (\a v) = \a\tmap
(v)$ and not just $\tmap (\a v) \lmodh \a \tmap (v)$. In fact,
these are equivalent when $\a \in \tT$, since $F$ is a
supertropical semifield. Indeed, assume that $\tmap (\a v) \lmodh
\a \tmap (v)$ for any $\a  \in \tT$ and $v \in V$. Then also $\a
^{-1} \in \tT$. By hypothesis,
$$\a ^{-1}\tmap (\a v) \lmodh \a ^{-1} \a \tmap (v)= \tmap(v)$$ and
$$\tmap (v) = \tmap (\a ^{-1} \a  v) \lmodh \a ^{-1}\tmap (\a  v),$$ so by
antisymmetry, $\a ^{-1}\tmap (\a v) = \tmap(v),$ implying $\tmap
(\a v) = \a \tmap(v).$\end{rem}

\begin{lem}\label{gsur} If $v \lmodg w$ then $\tmap(v)\lmodg
\tmap(w).$   \end{lem}
\begin{proof} Write $v = w + w'$ where $w' \in \tHz.$ Then  $$\tmap(v) \lmodg \tmap(w) + \tmap(w') \lmodg \tmap(w).$$
\end{proof}

\begin{lem}\label{gsur1} If $v \ge_\nu w$, then $\tmap(v)\ge_\nu
\tmap(w).$   \end{lem}
\begin{proof} By definition,  $v ^\nu \ge w^\nu ,$
implying $\tmap(v)^\nu = \tmap(v^\nu)\ge_\nu \tmap(w^\nu)=
\tmap(w)^\nu.$ (Since they are ghosts, $\tmap(v^\nu)\lmodg
\tmap(w^\nu)$ is the same as $\tmap(v^\nu)\ge_\nu \tmap(w^\nu)$.)
\end{proof}

\begin{prop} If $V=F$, then any supertropical map  $\tmap : V \to V'$ is strict.
\end{prop}
\begin{proof}  We need to show that $\tmap(a+b) = \tmap(a)+\tmap(b)$ for all $a,b \in V.$
First assume that $a >_\nu b$. Then $\tmap(a)\ge _\nu \tmap(b)$.
But $\tmap(a+b)= \tmap(a)$. If $\tmap(a)> _\nu \tmap(b)$, then
$\tmap(a+b)= \tmap(a)+\tmap(b)$. If $\tmap(a)\nucong \tmap(b)$,
then $$\tmap(a) = \tmap(a+b)\lmodg \tmap(a)^\nu,$$ implying
$\tmap(a) \in \tHz'$ and $ \tmap(a+b)= \tmap(a)^\nu
=\tmap(a)+\tmap(b).$

Thus, we may assume that  $a \nucong b$. But then $$\tmap(a+b)=
\tmap(a)^\nu = \tmap(a)+ \tmap(b)$$ since $\tmap(a) \nucong
\tmap(b)$ by Lemma~\ref{gsur1}.
\end{proof}
\begin{rem} There are two advantages that strict supertropical
maps have over supertropical maps. First,
  $(\tmap(V), \tmap(\tHz))$ is a
  submodule of $(V',  \tHz ')$, for any strict supertropical
 map $\tmap :
 V \to V'$, whereas this may not be so for other supertropical
 maps.

 Secondly,
any strict supertropical  map from $\Fz  ^{(n)} \to \Fz  ^{(n)}$
is defined up to ghost surpassing by its action on the standard
base. In particular, the strict supertropical map $\tmap : \Fz
^{(n)}\to \Fz ^{(n)}$ can be described  in terms of $n \times n$
matrices over $F$.  (Proposition \ref{prop:changebase1} shows that
when these maps are onto, the corresponding matrices are
generalized permutation matrices.) Any supertropical map agreeing
with $\tmap$ on the standard base must  ghost surpass  $\tmap$, so
in this sense the  strict supertropical maps are the ``minimal''
supertropical maps with respect to ghost surpassing.
\end{rem}

\begin{defn}\label{def:ghostKer}
Given a supertropical map   $\tmap : V \to V'$ of modules with
ghosts, we define the \textbf{ghost  kernel} $$\Gker (\tmap) :=
\tmap^{\invr}(\tHz') = \{ v\in V: \tmap(v) \in \tHz'\},$$ an
$R$-submodule   of $V$. We say that $\tmap$ is \textbf{ghost
monic} if $\tmap^{\invr}(\tHz') = \tHz.$
\end{defn}

\begin{defn}\label{def:onto}
A  supertropical  map $\tmap  : V \to W$ of vector spaces of rank
$n$ is called
 \textbf{tropically onto} if $\tmap(V)$ contains a d-base of $W$ of rank $n$.
 An \textbf{iso} is a supertropical map that is both ghost monic and
  tropically onto. (Note this need not be an isomorphism in the
  usual sense, since $\tmap$ need not be onto.)
\end{defn}

\begin{rem} The composition of isos is an iso, in view of Remark~\ref{def:thick0}.
\end{rem}



\subsubsection{Linear functionals}

\begin{defn} Suppose $V= (V,\tHz)$ is a vector space
over a supertropical semifield $F$. The set of supertropical maps
$$V^* := \Hom (V, \Fz ),$$
  is called the
\textbf{dual $F$-module} of $V$, and  its elements are called
\textbf{linear functionals};   i.e., any linear functional  $\lfun \in V^*$ satisfies
$$\lfun(v_1 + v_2)  \lmodh \lfun(v_1)  + \lfun(v_2) , \qquad  \lfun(a v_1)
= a \lfun(v_1), \qquad \ell(\tHz) \subseteq \tGz$$ for any $v_1,
v_2 \in V$ and $a \in \Fz $.
\end{defn}

A linear functional $\lfun: V \to \Fz $ is called a \textbf{ghost
functional}  if $\lfun(V) \in \tGz$ for all $v \in V$. We write
$\tHz^* \subset  V^*$ for the subset    of all the ghost linear
functionals; this is the ghost submodule of $V^*$. $(V^*, \tHz^*,
\nu^*)$ is a supertropical module  over $F$, under the natural
operations
$$(\ell_1 + \ell_2)(v) = \ell_1(v) + \ell_2(v), \qquad (a\ell)(v) = a\ell(v), \qquad \nu^*\ell(v) = \ell(v)^\nu,$$
for $ a \in F$, $v \in V$.
\subsection{Linear functionals on subspaces of $F^{(n)}$}

The idea here is to develop a theory of linear functionals for
$n$-dimensional subspaces $V\subseteq \Fz  ^{(n)}$
(Often $V=\Fz ^{(n)}$).

Towards this end, we want a definition of linear functionals that
respects a given d-base $\tB = \{b_1, \dots, b_n\}$  of $V$. We
define the matrix $$A^\Bnb := A^\nb A A^\nb,$$ cf. \cite[Remark
2.14]{IzhakianRowen2009Equations}, and recall that $I_A = A A^\nb$
and $I'_A = A^\nb A$ as defined in  Equations \eqref{eq:Anb} and
\eqref{eq:IA}.
 Since the elements of $\tB$ are tropically
independent, the matrix $A= A(\tB)$ is nonsingular, and so are the
matrices $$I_A = A A^\nb, \quad A^\Bnb = A^\nb A A^\nb = A^\nb
I_A, \quad  \text{and } \  I'_A = A^\Bnb A,$$  as well as $I_A A$
(since $I_A A A^\nb = I_A^2 = I_A$ is nonsingular).

\begin{defn} A d-base $\tB$ is \textbf{closed} if $I_A \tB =
\tB.$\end{defn}

There is an easy way to get a closed d-base from an arbitrary
d-base $\tB.$ From now on we set the matrix $$ A : = A(\tB).$$

\begin{defn}\label{sdbase0} Write $\oA = I_A A$, and let $\otB$ denote the rows of $\oA.$
 Let $$V_\tB : = \{ \oA v : v \in V
\},$$   the thick subspace of $V$ spanned  by $ \otB.$ \end{defn}

$V_\tB$ is the subspace of interest for us, since it is invariant
under the action of the matrix $A$.

\begin{rem} $\otB$ is obviously spanned by $\tB,$ but since $I_A A$ is
nonsingular, $\otB$ also is a d-base of $V$, and clearly $\otB$ is
closed since $I_A^2 =I_A$. Thus, $\tB$ is a d,s-base of $V_\tB.$
\end{rem}

The d-base $\otB$ easier to compute with, since now we have
$$I_{\oA} \oA = \oA.$$ From now on, replacing $\tB$ by $\otB$ if
necessary, we assume that the d-base $\tB$ of $V$ is closed.

 Rather than
dualizing all of $V$, we turn to the space
$$V^*_{\tB} : = \Hom (V_\tB, \Fz ).$$

 Define $L_A \in \Hom (V,V)$ by
$$L_A(v) := A^\Bnb v.$$
We also define the map $\tilL_A: V \to V $ by
$$\tilL_A(v) :=  I_A v.$$

\begin{rem}
$(\tilL_A)^2 = \tilL_A$, and $ \tilL_A$ is the identity on
$V_{\tB}$ since $$I_A (I_A Av) = I_A ^2 A v = I_A  A v . $$

Likewise, $L_A(v) = A^\nb v$ for all $v \in V_\tB.$
\end{rem}

\begin{lem}If $\lfun \in V^*_{\tB}$, then  $\lfun = (\lfun \circ
\tilL_A )\big|_{V_{\tB}}$ on $V_{\tB}$. In other words,
$$V^*_{\tB} = \{ {(\ell \circ \tilL_A)} \big|_{\tB} :
 \ell \in V^*\}.$$  \end{lem}
\begin{proof}  Follows at once from the remark.\end{proof}

\begin{lem} $V^*_{\tB}$ is a supertropical vector space, whose ghost submodule
$\tHz(V^*_{\tB})$ is $\{  f |_ {V_{\tB}} : f \in  \tHz^*\}$.
\end{lem}
\begin{proof} Suppose $f' \in \tHz(V^*_{\tB}).$ Let $f = f'\circ \tilL_A  \in \tHz^*.$ Then
$f' = f |_ {V_{\tB}}.$ The other inclusion is obvious.
\end{proof}

 \begin{defn}
Given a (closed) d-base $\tB= \{ b_1, \dots, b_n\}$ of $V,$ define
 $ \ep_i: V_{\tB} \to \Fz $ by
 $$ \ep_i(v)    = {b_i} ^\trn  L_A( v),$$
the scalar product of $b_i$ and $A^{\nb}v.$ Also, define $\tB^* =
\{\ep_i : 1 \le i \le n\} $. \end{defn}

When $v$ is tangible, we saw in   Remark~ \ref{sat10} that $$v \
\hsim \ \sum _{i=1}^n  b_i \Inu{\ep_i(v)}$$ is a saturated
tropical dependence relation of $v$ on the $b_i$'s; this is the
motivation behind our definition.

\begin{rem}\label{rmk:linearFun}  $ $
\begin{enumerate} \eroman
    \item

  $ \ep_i$ is a linear functional. Also, by definition,
  $\ep_i(b_j)$ is the $i,j$ position of $A A^\nb = I_A,$ a
  quasi-identity, which implies
  $$\ep_i(b_i) = \rone; \qquad \ep_i(b_j)\in \tG ,\ \forall i\ne j.$$
  Hence, $$\sum _{i=1}^n \a _i \ep_i (b_j) \lmodg \a _j \ep_j (b_j)
  = \a _j.$$

 \item $\sum  b_i \ep_i (v) = A\left( \begin{matrix} \ep_1 (v) \\ \vdots \\  \ep_n (v)
\end{matrix}\right) = A A^\nb  v =
  v,$ for $v \in V_{\tB}$.

\end{enumerate}
  \end{rem}

\begin{thm}\label{dualbasethm} If $R$ is a supertropical semifield and $\tB$ is a closed s-base of $V$,
then $\{\ep_i : 1 \le i \le n\}$ is a closed s-base of
$V^*_{\tB}$.
\end{thm}
\begin{proof} For any $\lfun \in V_\tB^*$, we write $\a _ i = \lfun(b_i) $,   and then see from Remark
\ref{rmk:linearFun} that $\sum \a_i \ep_i  \lmodg \lfun $ on
$V_{\tB}$.

It remains to show that the $\{\ep_i : i = 1, \dots, n\}$ are
tropically independent.
 If $\sum_{i=1}^n \beta_i \ep_i$ were ghost for some $\beta_i \in \tTz$, we would have $\sum \beta_i  {b_i}^\trn A^\nb$ ghost.
Let $D$ denote the diagonal matrix  $\{\beta_1, \dots, \beta_n\},$
  and let  $\mathcal I = \{ i: \beta_i \ne \rzero \},$ and assume
  there are $k$ such tangible coefficients $ \beta_i$. Then  for any
  $i\notin \mathcal I$ we have $\beta_i = 0$, implying the $i$ row
 of the matrix  $D I_A$ is zero. But
  the sum of the rows of the matrix~$D I_A$ corresponding to indices from
  $\mathcal I$
 would be $\sum \beta_i  {b_i}^\trn   A^\nb$, which is ghost, implying that these $k$ rows of  $D I_A$
are dependent; hence  $D I_A$ has rank $\le k-1$. On the other
hand, the $k$ rows of  $D I_A$ corresponding to indices from
$\mathcal I$ yield a $k \times k$ submatrix of determinant $\prod
_{i \in \mathcal I} \beta_i \in \tT,$ implying its rank $\ge k$ by
\cite[Theorem 3.4]{IzhakianRowen2009TropicalRank}, a
contradiction.
   \end{proof}

  In the view of the theorem, we denote $\tB^* = \{\ep_i : 1 \le i \le n\}$, and call it the
  (tropical)
\textbf{dual s-base} of $\base$.




 Write $V^{**}_\tB$ for $(V^*_\tB)^*$. Define a map $$\Phi: V_\tB \to
V^{**}_\tB,$$  given by  $v \mapsto f_v$, where
$$f_v(\lfun) \ = \ \lfun(v).$$
\begin{rem} Let $v_j$ denote the $j$-th row of $A^\Bnb$, i.e., $v_j = b'_j$.
Since $A A^\nb = I_A$ is a quasi-identity matrix, we see that
$$f_{b_j}(\ep_i) = \ep_i(b_j) = {b_i}^\trn A^\Bnb \, b_j =
\left\{%
\begin{array}{ll}
    \frac{|A|}{|A|}b_i =b_i, & i =j ; \\
    \ghost, & i \neq j . \\
\end{array}%
\right.    $$
\end{rem}
\begin{lem}\label{fundep} Suppose $v = \sum \a_i b_i$,
for $\a_i \in \tT.$ Then  $f_v(\eps_i)\notin  \tHz$ for some $  i
.$\end{lem} \begin{proof}$(\Leftarrow)$ The assertion is obvious.
$(\Rightarrow)$ Suppose $\eps_i(v) = f_v(\eps_i)\in  \tHz$ for
each $i$. Then $\sum \a_i b_i \in \tHz,$ contrary to the $b_i$
being tropically independent.
\end{proof}

\begin{example}\label{6.23} Suppose $V= \Fz  ^{(n)}$, a supertropical vector space.  The map $\Phi: V \to V^{**}$ is a
vector space isomorphism when $\base$ is the standard base.
\end{example}

\begin{prop} For any $v\in V,$ define $v^{**}\in V^{**}$ by $v^{**}(\ell) = \ell(v).$ The map $\Phi: V_\tB \to V^{**}_\tB $ given by $v \mapsto v^{**}$ is an iso of
supertropical vector spaces.
\end{prop}
\begin{proof} $\Phi(\tB)$ is a d-base of $n$ elements, which is ghost injective, since  any non-ghost vector $v = \sum \a_i b_i$ of $\Phi(\tB)$ has some tangible coefficient   $\a _i$, and then $v^{**}(\eps_i) = \a_i \in \tT$. But
by Example~\ref{6.23}, taking the standard classical base, we see
that $V^{**}$ has  rank $n$. Hence any supertropical subspace
having $n$ tropically independent elements is thick.
\end{proof}


 \section{Supertropical bilinear forms}\label{bilform}

The classical way to study orthogonality in vector spaces is by
means of bilinear forms. In this section, we introduce the
supertropical analog, providing some of the basic properties.
Although the tropical literature deals with orthogonality in terms
of the inner product, as described in \cite[\S~25.6]{ABG}, the
supertropical theory leads to a more axiomatic approach.

 The notion of supertropical bilinear form follows the
classical algebraic theory, although, as to be expected, there are
a few surprises, mostly because of the characteristic 2 nature of
the theory \cite{IzhakianKnebuschRowen2010bilinear}. In this
section, we assume that $V$ is a vector space over a supertropical
semifield $F$.

 \subsection{Supertropical bilinear forms}

\begin{defn} \label{12} A (supertropical) \textbf{bilinear form} on supertropical vector spaces $V = (V, \tHz)$
and $V' = (V', \tHz')$ is a function $B : V\times V' \to \Fz$ that
is a linear functional in each variable; i.e., writing $(v,w)
\mapsto B(v,w)$ for $v\in V$ and $w\in V'$, any given $u$ in $V$
and $u' \in V'$, we have linear functionals
$$B( u,\underline{\phantom{w}}\,)\: w \mapsto
B(u,w),\qquad B(\underline{\phantom{w}  }\, ,u')\: v \mapsto
B(v,u'),$$ satisfying $B(V,\tHz') \subseteq \tGz$ and  $B(\tHz,V')
\subseteq \tGz$.  Thus, $$ B(v_1 +   v_2, w_1 +  w_2) \lmodg
B(v_1,w_1) + B(v_1,w_2) + B(v_2,w_1) + B(v_2,w_2) ,$$

$$
B(\a v, w) = \a  B(v ,w )=   B(v ,\a w), $$ for all $\a \in F$ and
$v_i\in V,$ and $w_j\in V'.$

 When $V' = V$, we say that
$B$ is a \textbf{(supertropical) bilinear form} on the vector
space $V.$ We say that a bilinear form $B$ is \textbf{strict} if
$$ \begin{array}{lll}
B(\a_1v_1 + \a_2 v_2, \bt_1w_1 + \bt_2 w_2)   = \\ \qquad
\a_1\bt_1 B(v_1,w_1) + \a_1\bt_2 B(v_1,w_2)  + \a_2\bt_1
B(v_2,w_1) +\a_2\bt_2 B(v_2,w_2) ,\end{array} $$ for all $v_i \in
V$ and $w_i \in V'$.
\end{defn}

We often suppress $B$ in the notation, writing $\langle v,w
\rangle$ for $B(v,w)$. Perhaps surprisingly, one can lift many of
the classical theorems about bilinear forms to the supertropical
setting, without requiring strictness.

\begin{example} There is a natural bilinear form $B:V\times V^* \to
\Fz$, given by $B(v,f) = f(v)$, for $v\in V$ and $f\in
V^*$.\end{example}

 \begin{rem}\label{natmap} $ $
\begin{enumerate} \eroman
    \item There is a natural map $\Phi: V' \to V^*, $ given by $w\mapsto
\bil {\underline{\phantom{w}}}w$. Likewise, there is a natural map
$\Phi: V \to (V')^*, $ given by $v\mapsto \bil
{\underline{\phantom{w}}}v$.

 \pSkip

    \item  For any bilinear form $B$, if $v \lmodh \sum \a _i v_i$ and $w
    \lmodh
\sum \beta _j w_j$, for $\a_i, \beta_j \in \tT,$ then
\begin{equation}\label{0.4} \bil vw \lmodg  \sum _{i,j }  \alpha_i \beta_j\bil {v_i}{w_j} .\end{equation}
\end{enumerate}
\end{rem}

 For the remainder of this
section, we take $V' = V \subseteq  \Fz^{(n)}$, a vector space
over the supertropical semifield $\Fz $, and consider a
(supertropical) bilinear form B on $V$.

\begin{defn} The \textbf{Gram matrix} of vectors $v_1, \dots,
v_k \in V = \Fz^{(n)}$ is defined as the $k \times k$  matrix
\begin{equation}\label{eq:GramMatrix}
\tilG(v_1, \dots, v_k ) = \left( \begin{array}{cccc}
                      \bil {v_1}{v_1} &  \bil {v_1}{v_2} & \cdots & \bil {v_1}{v_k} \\
                      \bil {v_2}{v_1} &  \bil {v_2}{v_2} & \cdots & \bil {v_2}{v_k} \\
                         \vdots & \vdots &  \ddots & \vdots \\
                      \bil {v_k}{v_1} &  \bil {v_k}{v_2} & \cdots & \bil {v_k}{v_k} \\
                        \end{array} \right).
\end{equation}
The set $\{v_1, \dots, v_k\}$ is \textbf{nonsingular} (with
respect to $B$) iff its Gram matrix is nonsingular (see
\S\ref{ses:matrices}).

 The \textbf{Gram
matrix} of $V$ is the Gram matrix of an s-base of $V$.
\end{defn}

\begin{example}\label{ex1} The quasi-identity  \begin{equation}\label{eq:GramMatrix1}
\tilG(v_1 , v_2 ) = \left( \begin{array}{cc}
                    0 & 1^\nu \\
                     -\infty & 0 \\
                        \end{array} \right)\end{equation} (in logarithmic
                        notation) is the Gram matrix of a
                       bilinear form.
                        Note that $$\bil{v_1}{v_2} = 1^\nu > 0^\nu
                        = \bil{v_1}{v_1}+ \bil{v_2}{v_2}.$$
\end{example}

In particular, we have the matrix $\widetilde G = \tilG(b_1 \dots,
b_k ) $, which can be written as $(g_{i,j})$ where $g _{i,j} =
\bil {b_i}{b_j};$  \eqref{0.4}  written in
 matrix
notation becomes \begin{equation}\label{0.5} \bil vw \lmodg v^\trn
\, \tilG w .\end{equation}
 Of course, the
matrix $\tilG$ depends on the choice of tangible  s-base $\tB$ of
$V$, but this is unique up to multiplication by scalars and
permutation, so $\tilG$ is unique up to $P \tilG P^t$ where $P$ is
a generalized permutation matrix. In particular, whether or not
$\tilG$ is nonsingular does not depend on the choice of  s-base.

\subsection{Ghost orthogonality}

Bilinear forms play a key role in geometry since they permit us to
define orthogonality of supertropical vectors. However, as we
shall see, orthogonality is   rather delicate   in this setup.

\begin{defn} We write
 $v \gperp w$
when $\bil{v}{w} \in \tGz$,  that is $\langle w_1,w_2\rangle\lmodg
\fzero$ (cf. Remark \ref{rem:ghostSrp}), and say that $v$ and $w$
are   \textbf{ghost orthogonal}, or   \textbf{g-orthogonal} for
short. Likewise, subspaces $W_1,$ $W_2$ of $V$ are
\textbf{g-orthogonal} if $\langle w_1,w_2\rangle\in \tGz$ for all
$w_i\in W_i.$

A subset $S$ of $V$ is \textbf{g-orthogonal} (with respect to a
given bilinear form) if any pair of distinct vectors from $S$ is
g-orthogonal. The \textbf{(left) orthogonal ghost complement} of
$S$ is defined as
$$\Sp = \{ v\in V: \bil vS \in \tGz \}.$$
\end{defn}

The orthogonal ghost complement $ \Sp$ of any set $S \subset V$ is
a subspace of $V$, and $\tHz \subseteq \Sp $ for any $S \subset
V.$ Note that g-orthogonality is not necessarily a symmetric
relation.

\begin{defn} A subspace $W$ of $V$ is called \textbf{nondegenerate} (with
respect to $B$), if  $W^\gperp \cap W  \subseteq \tHz$. The
bilinear form $B$ is \textbf{nondegenerate} if the space $V$ is
 nondegenerate.

 The \textbf{radical}, $\rad( V)$, with respect to a given bilinear form $B$,
  is defined as $V^\gperp.$ Vectors $w_i$ are \textbf{radically
dependent} if $\sum \a _i w_i \in \rad (V)$ for suitable $\a_i \in
\tTz,$ not all $\fzero.$
\end{defn}

Clearly, $\tHz \subseteq \rad(V)$.

\begin{rem} \label{15} $ $
\begin{enumerate}\eroman
    \item $\rad (V) = \tHz$ when
$V$ is nondegenerate, in which case radical dependence is the same
as tropical dependence. \pSkip
    \item

Any ghost complement $V'$ of $\rad (V)$ is obviously tropically
g-orthogonal to $\rad (V),$ and nondegenerate since $$\rad (V') \
\subseteq \ V'  \cap  \rad (V) \ \subseteq \tHz.$$  This
observation enables us to reduce many proofs to nondegenerate
subspaces, especially when a  Gram-Schmidt procedure  is
applicable (to be described in
\cite{IzhakianKnebuschRowen2010bilinear}).
\end{enumerate}
\end{rem}

\begin{lem}\label{Gramrek} Suppose $\{w_1, \dots, w_m \}$ tropically span
a subspace $W$ of $V$. If $ \sum \bt_i \bil v{w_i} \in \tGz$  for
each $v\in V,$ then
 $\sum_{i=1}^m \bt_i w_i \in W^\gperp.$
\end{lem}
\begin{proof}  $\bil v {
\sum_i \bt_i  w_i} \lmodg \sum_i \bil v {\bt_i w_i} =\sum_i \bt_i
\bil v { w_i}\in \tGz$ for all $v\in W$. Thus, $\sum_i \bt_i w_i
\in W^\gperp$.
\end{proof}

\begin{thm}\label{thm:GramMat} Assume that vectors
 $w_1 \dots, w_k \in V$  span  a nondegenerate subspace $W$ of $V$.
If $ | \tilG(w_1 \dots, w_k ) | \in \tGz$, then $w_1 \dots, w_k $
are tropically dependent.
\end{thm}
\begin{proof}  Write $\tilG = \tilG(v_1, \dots, v_k ).$ By ~\cite[Theorem 6.6]{IzhakianRowen2008Matrices},
 $|\tilG| \in \tGz$ iff the rows of $\tilG$
are tropically dependent. By the lemma, if $|\tilG| \in \tGz$,
then some linear combination of the $v_i$ is in $W^\gperp$. When
$W$ is nondegenerate, this latter assertion is the same as saying
that the $v_i$ are tropically dependent.
\end{proof}

\begin{cor}\label{nondeg1}
If the bilinear form $B$ is nondegenerate on a vector space $V$,
then   the Gram matrix   (with respect to any given supertropical
d,s-base of $V$) is nonsingular.
\end{cor}

\begin{rem} In case the bilinear form $B$ is strict, we can strengthen Lemma~\ref{Gramrek} to
obtain:
$$\sum_{i=1}^m \bt_i w_i \in W^\gperp \qquad \text { iff } \qquad  \sum
\bt_i \bil v{w_i} \in \tGz$$ for each $v\in V.$ (Indeed, if
$\sum_{i} \bt_i w_i \in  W^\gperp$, then  $\sum_i \bt_i \bil v
{w_i}  = \bil v { \sum_i \beta_i w_i}  \in \tGz$ for all $i$.)
\end{rem}

 In this case, we can also strengthen Corollary~\ref{nondeg1} to read:

\begin{cor}\label{nondeg2} A strict bilinear form $B$ is nondegenerate on a supertropical vector space $V$ iff
the Gram matrix   (with respect to any given supertropical
d,s-base of $V$) is nonsingular.
\end{cor}

\subsection{Symmetry of g-orthogonality}

In this subsection, we prove the supertropical version of a
 classical theorem of Artin, that any bilinear form
in which g-orthogonality is symmetric must be either an alternate
or symmetric bilinear form.
 In characteristic 2, any alternate
 form is symmetric, so we would expect our supertropical forms to
 be symmetric in some sense.

 \begin{defn} The (supertropical) bilinear form $B$ is \textbf{orthogonal-symmetric} if it satisfies the property for all $v_i,w
\in V$:

\begin{equation}\label{orthsym} \qquad \sum _i \bil{v_i}{w } \in \tGz \quad  \text{iff}\quad
\sum _i \bil{w }{v_i} \in \tGz, \end{equation}
 \noindent for any finite sum taken over $v_i \in V$.

 $B$ is \textbf{supertropically symmetric}  if $B$ is
orthogonal-symmetric and satisfies the additional property that
$\bil vw \nucong \bil wv $ for all $v,w \in V$ satisfying $\bil vw
\in \tT.$

A vector $v \in V$ is \textbf{ isotropic}  if $\bil v v \in \tGz$;
the vector $v$ is \textbf{strictly isotropic}  if $\bil v v =
\fzero$.
 \begin{rem}\label{vertriv} If every $v\in V$ is strictly isotropic,  then the (supertropical) bilinear form $B$ is
 trivial. (Indeed, $$ \fzero = \bil {v+w}{v+w} = \bil {v}{w}+\bil {w}{v}$$ for all
$v,w \in V,$ implying $\bil {v}{w}=\bil {w}{v} = \fzero.$)
\end{rem}
\end{defn}

 \begin{rem}\label{verify0} When the bilinear form $B$ is strict,
Condition~\eqref{orthsym} reduces to the condition

 $$ \bil{v}{w } \in \tGz \quad  \text{iff}\quad \bil{w
}{v} \in \tGz$$ since, taking $v = \sum _i v_i,$ we have
$$ \sum _i \bil{v_i}{w } =  \bil{v}{w }; \qquad  \bil{w}{v } = \sum _i \bil{w }{v_i}.$$

\end{rem}


In general, we need Condition~\eqref{orthsym} to carry through the
proof of Theorem~\ref{orthogsym} below.

 \begin{lem}\label{verify1}
An orthogonal-symmetric bilinear form  $B$ is supertropically
symmetric if it satisfies the condition that $\bil vw + \bil wv\in
\tGz$ for all vectors $v,w \in V$.
\end{lem}
 \begin{proof}
 If $\bil vw\in\tGz$, then $ \bil wv~\in~\tGz $ by orthogonal-symmetry. Thus, we may assume
  that $\bil vw \in
 \tT$. But then $\bil wv \in
 \tT$ by orthogonal-symmetry; by hypothesis, $\bil
vw + \bil wv\in \tGz$, implying $\bil vw \nucong \bil wv $, as
desired.
\end{proof}

Also, the symmetry condition extends to sums.

 \begin{lem}\label{verify2} If $B$ is supertropically symmetric,
 then $$\sum _i \bil{v_i}{w } \in \tT \quad\text{ iff } \quad \sum _i \bil{w
}{v_i} \in \tT$$ with $\sum _i \bil{v_i}{w }= \sum _i \bil{w
}{v_i} \in \tT$.
 \end{lem}
 \begin{proof} We may assume that $\sum _i \bil{v_i}{w }, \sum _i \bil{w
}{v_i} \in
 \tT$, since there is nothing to check if one (and thus the other) is ghost.
 Take $i_1$ such that $\bil{v_{i_1}}{w }$ is the dominant summand of $\sum _i\bil{v_i}{w }$,
 and thus is tangible. Likewise, take $i_2$ such that $\bil{w }{v_{i_2}}$
 is the dominant summand of $\sum _i \bil{w }{v_i}$,
 and thus is tangible. By hypothesis $\bil{v_{i_1}}{w }
 = \bil{w}{v_{i_1}}$ and $\bil{w}{v_{i_2}} = \bil{v_{i_2}}{w }
 $.  Since these dominate their respective sums,
we get $\sum _i \bil{v_i}{w }= \sum _i \bil{w }{v_i} \in \tT$.
  \end{proof}

We aim to prove that an orthogonal-symmetric (supertropical)
bilinear form is supertropically
 symmetric.

Another important property to check is when $\bil{v}{w}+\bil{w}{v}
\in \tGz.$ This  condition means that $v$ is orthogonal to $w$
with respect to the new bilinear form given by $\bil{v}{w}' : =
\bil{v}{w}+\bil{w}{v} ,$ and arises here in several assertions.

\begin{lem}\label{zero1} Suppose that $B$ is an
orthogonal-symmetric  bilinear form and $v,w \in V$. Then either
$\bil{v}{w}+\bil{w}{v} \in \tGz$, or $v$ and $w$ are strictly
isotropic.
\end{lem}\begin{proof} One may assume that $\bil vw \in \tT$; hence  $\bil wv \in \tT$.
If $ \bil vw \nucong  \bil wv$ then $\bil{v}{w}+\bil{w}{v} \in
\tGz$, so we may assume by symmetry that $ \bil vw
>_\nu  \bil wv$.

First assume that $w$ is nonisotropic. Then   $\gamma \bil vw
+\bil ww$ is ghost for $\gamma= \frac {\bil ww} {\bil vw}$ and
tangible for  any other tangible $\gamma$ in $F$. But $\gamma \bil
wv +\bil ww $  is ghost for $\gamma = \frac {\bil ww}{\bil wv}$,
contradicting orthogonal-symmetry unless $\bil vw \nucong \bil
wv,$ implying $\bil vw + \bil wv \in \tGz.$

Next assume that $w$ is isotropic but $\bil ww = \a^\nu \ne
\fzero$ for $\a \in \tT$. Then for tangible $\gamma >_\nu \frac
{\bil ww} {\bil vw}$ we see that $\bil {\gamma v}w + \bil w{ w}$
is tangible, so $\bil w{\gamma v}+ \bil w{ w}$ must also be
tangible, which is false if $\gamma <_\nu \frac {\bil ww} {\bil
wv}$. This yields a contradiction if ${\bil wv}<_\nu {\bil vw}$,
and similarly we have a contradiction if ${\bil wv}>_\nu {\bil
vw}$; hence ${\bil wv} \cong_\nu {\bil vw},$ implying $\bil vw +
\bil wv \in \tGz.$

Thus, we may assume that  $\bil ww = \fzero$. Likewise, $\bil vv =
\fzero$, since otherwise we would conclude by interchanging $v$
and $w$.
\end{proof}

We conclude with our supertropical version of Artin's theorem.

\begin{thm}\label{orthogsym} Every orthogonal-symmetric bilinear form~$B$ on a supertropical vector space
$V$  is  supertropically symmetric.
\end{thm}
\begin{proof}  We are done by Lemma~\ref{zero1} unless there are  vectors $v,w \in V$ for
which $\bil vv = \bil ww = \fzero$ and $\bil vw +\bil wv\in \tT$.

$\a := \bil vw \in \tT$; then $\beta :=\bil wv \in \tT,$ and  $\a
+ \beta \in \tT.$ Observe that, if $v'\in V$ such that $\bil{v'}w
\nucong \a,$ then $\bil w{v'} \nucong \beta$. Indeed, $\bil{v}w
+\bil{v'}w  = \a^\nu,$ implying $\bil w{v}+\bil w{v'} \in \tG.$
But $\bil w{v'}\in \tT,$ so we conclude that $\bil w{v'} \nucong
\beta$.

Now let vector $v'$ be any vector of $ V$. Then  $$\bil {v+v'}w
\lmodg \bil vw + \bil {v'}w \ne \fzero.$$ Thus, $ \bil{v+v'}w
\nucong \gamma$,  for some $\gamma\in \tT$.  Let $v'' := \frac
{\a}{\gamma} (v+v')$. Then $$\bil{v''}w = \frac {\a}{\gamma}\bil{
v+v'}w \nucong \a,$$ and thus $\bil w{v''} \nucong \beta,$ as just
observed. Hence, $\bil{v''}w + \bil w{v''}\notin \tG.$ Now
Lemma~\ref{zero1} yields $\bil{v''}{v''} = \fzero.$ From
$$\fzero = \bil {\gamma v''}{\gamma v''} \lmodg \bil {v}{v}+ \bil
{v}{v'}+ \bil {v'}{v}+\bil {v'}{v'},$$ we conclude that $\bil
{v'}{v'}= \fzero$ for all $v'\in V$; i.e., $B$ is trivial, by
Remark~\ref{vertriv}, which is absurd since $\a = \bil vw \ne
\fzero$. Thus, $B$ must be supertropically symmetric.
\end{proof}


\bibliographystyle{abbrv}


\end{document}